\documentclass{amsart}
\usepackage{amsmath}
\usepackage{amsthm}
\usepackage{amsfonts}
\usepackage{amssymb}
\usepackage{graphicx}
\usepackage{enumerate}
\usepackage{color}
\usepackage[english]{babel}
\usepackage[utf8]{inputenc}
\usepackage{hyperref}

\begingroup\expandafter\expandafter\expandafter\endgroup
\expandafter\ifx\csname pdfsuppresswarningpagegroup\endcsname\relax
\else
\fi

\newtheorem{theorem}{Theorem}[section]
\theoremstyle{plain}

\newtheorem{corollary}[theorem]{Corollary}

\newtheorem{definition}[theorem]{Definition}

\newtheorem{lemma}[theorem]{Lemma}

\newtheorem{proposition}[theorem]{Proposition}

\theoremstyle{definition}
\newtheorem{remark}[theorem]{Remark}
\newtheorem{example}[theorem]{Example}

\numberwithin{equation}{section}

\DeclareMathOperator{\id}{id}
\DeclareMathOperator{\dimension}{dim}

\newcommand{\R}{\mathbb{R}}

\newcommand{\defterm}[1]{\textbf{#1}}
\def \twistmap {\mathcal{D}}
\def \submonodromy {\psi}

\newcommand{\subopenbook}{nested open book}
\newcommand{\subpage}{nested page}
\newcommand{\subbinding }{nested binding}

\makeatletter
\renewcommand{\@biblabel}[1]{[#1]\hfill}
\makeatother

\begin{document}

%--------------------------------------------------------
\title{Nested Open Books And The Binding Sum}
\author{Sebastian Durst}
\address[S.~Durst]{
Fraunhofer-Institut f\"ur Hochfrequenzphysik und Radartechnik FHR\\
Fraunhoferstr. 20\\ 53343 Wachtberg\\ Germany}
\email{sdurst@math.uni-koeln.de}
\author{Mirko Klukas}
\address[M.~Klukas]{MIT Department of Brain and Cognitive Sciences\\ 43 Vassar St.\\ Cambridge MA 02139\\ USA}
\email{mirko.klukas@gmail.com}
\subjclass[2010]{57R17 (primary), and 57R65 (secondary)} 
%--------------------------------------------------------

\selectlanguage{english}

\begin{abstract}
We introduce the notion of a \emph{\subopenbook}, a submanifold equipped with
an open book structure compatible with an ambient open book, and describe in detail
the special case of a \emph{push-off} of the binding of an open book. This enables
us to explicitly describe a natural open book decomposition of a fibre connected sum of
two open books along their bindings, provided they are diffeomorphic and 
admit an open book structure themselves.
Furthermore, we apply the results to contact open books, showing that the natural
open book structure of a contact fibre connected sum of two adapted open books
along their contactomorphic bindings is again adapted to the resulting contact structure.
\end{abstract}

\maketitle

%%%%%%%%%%%%%%%%%%%%%%%%%%%%%%%%%%%%%%%%%%%%%%%%%%%%%%%%%%%%%
% 
\section{Introduction}
% 
%%%%%%%%%%%%%%%%%%%%%%%%%%%%%%%%%%%%%%%%%%%%%%%%%%%%%%%%%%%%%
In 1923 Alexander \cite{Alexander1923} proved that every closed oriented $3$-manifold admits a so-called \textit{open book decomposition}. 
In fact, combining the work of Winkelnkemper, Lawson, and Quinn from the 1970s, this statement remains true for odd-dimensional manifolds in general (see \cite{MR0310912}, \cite{MR0494132}, and \cite{MR528236} respectively). The existence problem in even dimensions is also
solved in these works but is more involved.
Arguably the most prominent appearance of open books in recent history 
is their surprisingly deep connection with \textit{contact structures}.
As was observed by Giroux \cite{Giroux} in 2002, contact structures in dimension 3 are
of purely topological nature: he established a one-to-one correspondence between
isotopy classes of contact structures and open book decompositions up to positive
stabilisation. This correlation remains partly intact in higher dimensions.
According to Giroux and Mohsen \cite{GirouxMohsen} 
any contact structure on a closed manifold of dimension at least 3 
admits a \textit{compatible} open book decomposition. 
In particular, its binding inherits an induced contact structure and thus possesses a compatible open book itself.

In this paper we investigate how the \textit{fibre connected sum} performed along diffeomorphic binding 
components of an open book, henceforth called \textit{binding sum}, affects an underlying open book structure. 
We discuss both regular fibre sums and fibre sums of contact manifolds.
While, at first glance, the binding sum destroys the open book structure, we show that this is in fact not the case and generalise a previous result of the second author \cite{Klukas} to higher dimensions.
The main results can be summarized as follows:

\begin{list}{}{}
\item[\textbf{I.}] \textbf{Existence:} The fibre connected sum along diffeomorphic binding 
components of an open book admits a \textit{natural} open book structure, provided the binding components admit open book structures themselves (see Theorem~\ref{thm:main1} for a detailed statement).
\item[\textbf{II.}] \textbf{Compatibility:} 
Our construction can be adapted to an underlying contact structure and again produces a compatible open book (see \ Theorem~\ref{thm:contact}).
\end{list}

The open book is \textit{natural} in the sense that
it will be described in terms 
of the original open book, and a fixed open book decomposition of the
binding components %.
(see Theorem~\ref{thm:main1} for details).
The idea of the construction is to form the fibre connected sum not along the
binding components themselves but along slightly isotoped copies, called
the \emph{push-offs}, realising them as \emph{nested open books}.
A nested open book is a submanifold inheriting an open book structure from
the ambient manifold and is thus a natural generalisation of a
\emph{spinning} as discussed in contact topology by
Mori \cite{Mori2004} and Mart\'inez Torres \cite{Torres2011}.
A nice survey on topological spinnings is \cite{Friedman2005}.

\begin{remark}
The first result remains true if the two binding components admit fibrations over the circle (that is, informally, if they admit similar open book decompositions with binding the empty set). In particular, the statement holds in the $3$-dimensional case, where the binding components are circles. 
\end{remark}

Note that, according to the above mentioned work of  Winkelnkemper, Lawson, and Quinn, the condition on the binding components to admit open books themselves is not a restriction in odd dimensions. For the same reasons
the sheer existence of an open book of the fibre sum in odd dimensions is already answered as well.
According to Giroux and Mohsen the same is true for the case of adapted open books and contact structures.
However, constructions of particular instances of those open books and their relationship to the original open book decomposition have rarely been discussed in the literature yet.
One notable exception is Mori's construction of contact structures and leaf-wise symplectic foliations on $S^4\times S^1$ arising as the fibre connected sum of two copies of $S^5$ (\cite{Mori2012}, see Section~\ref{section:contact_examples}).
Besides, very few constructions of open books supporting contact structures in higher dimensions are known.
Along these lines we explain how binding sums can be utilised to describe 
fibrations over the circle whose fibres are convex hypersurfaces in the sense of Giroux, and manifolds that admit 
the higher-dimensional analogue of \textit{Giroux torsion} introduced by Massot, Niederkr\"uger and Wendl \cite{MNW13} (cf.\ Section~\ref{section:contact_examples}). Our construction thus yields open book decompositions for both of these two classes of contact manifolds.

\subsection{Acknowledgments}
%\begin{acknowledgements}
We thank an anonymous referee for bringing Mori's construction to our attention.
%\end{acknowledgements}

%%%%%%%%%%%%%%%%%%%%%%%%%%%%%%%%%%%%%%%%%%%%%%%%%%%%%%%%%%%%%
% 
\section{Preliminaries}
% 
%%%%%%%%%%%%%%%%%%%%%%%%%%%%%%%%%%%%%%%%%%%%%%%%%%%%%%%%%%%%%

\subsection{Open book decompositions}
\label{section:prelim_open_books}

An \defterm{open book decomposition} of an $n$-dimensional manifold $M$ is a pair
$(B,\pi)$, where $B$ is a co-dimension $2$ submanifold in $M$, called the
  \defterm{binding} of the open book, and
  $\pi\colon\thinspace M\setminus B \to S^1$ is a (smooth, locally trivial)
  fibration such that
  each fibre $\pi^{-1}(\theta)$, $\theta\in S^1$, corresponds to the interior
  of a compact hypersurface $\Sigma_\theta \subset M$ with
  $\partial\Sigma_\theta$ equal to $B$,
	and the binding has a tubular neighbourhood which is trivialised by $\theta$.
  The hypersurfaces $\Sigma_\theta$, $\theta \in S^1$, are called the
  \defterm{pages} of the open book. We will say that $M$ admits an \textbf{open book structure}.
\parskip 0pt

\begin{figure}[htb]
\includegraphics[width=1.\textwidth]{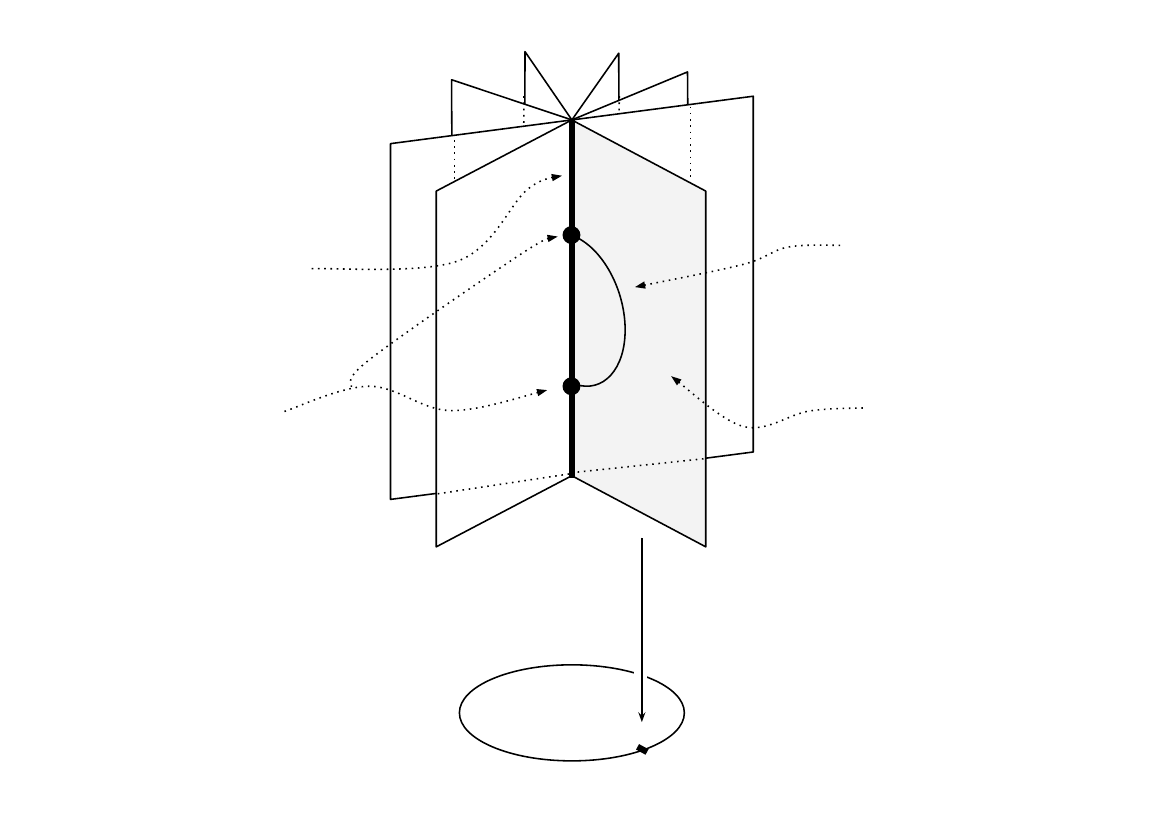}
\put(-275, 165){$B$}
\put(-286, 120){$B'$}
\put(-93, 121){$\Sigma_\theta$}
\put(-100, 170){$\Sigma'_\theta$}
\put(-157, 53){$\pi$}
\put(-157, 12){$\theta$}
\caption{Schematic picture of an open book. A single page of a nested open book and its nested binding is indicated.}
 \label{fig:ob_sphere_disc}
\end{figure}
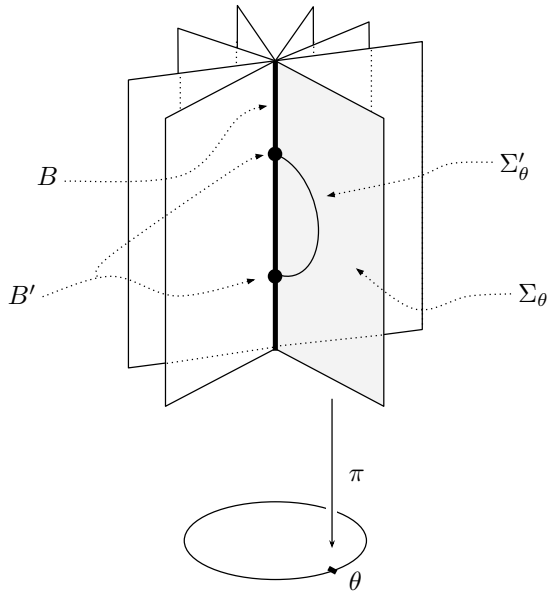

The question of which data is relevant to remodel the ambient manifold and the underlying open book structure up to diffeomorphism leads us to the following notion.

An \defterm{abstract open book} is a pair $(\Sigma,\phi)$, where $\Sigma$ is a compact manifold with non-empty boundary $\partial \Sigma$, called the \defterm{page}, and $\phi\colon\thinspace \Sigma \to \Sigma$ is a diffeomorphism equal to the identity near $\partial \Sigma$, called the \defterm{monodromy} of the open book.
Let $\Sigma(\phi)$ denote the mapping torus of $\phi$, that is, the quotient space obtained from $\Sigma \times [0,2\pi]$ by identifying $(x, 2\pi)$ with $(\phi(x),0)$ for each $x \in \Sigma$. Then the pair $(\Sigma,\phi)$ determines a closed manifold $M_{(\Sigma,\phi)}$ defined by
\begin{equation}
\label{eqn:abstract open book}
			M_{(\Sigma,\phi)} := \Sigma(\phi) \cup_{\id} (\partial \Sigma \times D^2),
\end{equation}
where we identify $\partial \Sigma(\phi) = \partial \Sigma \times S^1$ with $\partial (\partial \Sigma \times D^2)$ using the identity map.
Let $B \subset M_{(\Sigma,\phi)}$ denote the embedded submanifold $\partial \Sigma \times \{0\}$. Then we can define a fibration $\pi\colon\thinspace M_{(\Sigma,\phi)}\setminus B \to S^1$ by
\[
	\left.
	 \begin{array}{l}
				\lbrack x,\theta \rbrack  \\
				 \lbrack x', r \textup{e}^{i\theta} \rbrack
		\end{array} \right\}
							\mapsto [\theta],
\]
where we understand $M_{(\Sigma,\phi)}\setminus B$ as decomposed as in (\ref{eqn:abstract open
book}) and $[x,\theta] \in \Sigma(\phi)$ or $ [x', r\textup{e}^{i\theta}] \in
\partial\Sigma \times D^2 \subset \partial\Sigma \times \mathbb{C}$. Clearly, $(B,\pi)$ defines an open book decomposition of $M_{(\Sigma,\phi)}$.
\parskip 0pt

On the other hand, an open book decomposition $(B,\pi)$ of some $n$-manifold
$M$ defines an abstract open book as follows: identify a neighbourhood of $B$
with $B \times D^2$ such that $B = B\times\{0\}$ and such that the fibration on
this neighbourhood is given by the angular coordinate, $\theta$ say, on the
$D^2$-factor. We can define a $1$-form $\alpha$ on the complement $M \setminus (B \times
D^2)$ by pulling back $d\theta$ under the fibration $\pi$, where this time we
understand $\theta$ as the coordinate on the target space of $\pi$.
The vector field $\partial_\theta$ on $\partial\big(M \setminus (B \times D^2)
\big)$ extends to a nowhere vanishing vector field $X$ which we normalise by
demanding it to satisfy $\alpha(X)=1$. Let $\phi$ denote the time-$2\pi$ map of the
flow of $X$. Then the pair $(\Sigma,\phi)$, with $\Sigma = \overline{\pi^{-1}(0)}$, defines an abstract open book such that $M_{(\Sigma,\phi)}$ is diffeomorphic to $M$.

Nice surveys on open books and their applications are Winkelnkemper's appendix to
Ranicki's book \cite{Ranicki98} and Giroux \cite{GirouxWhatIs}. More detailed material, in particular on the relation of open books and contact structures can be found in \cite{Etnyre2004,Geiges2008,Koert2010}.

\subsection{The fibre connected sum and the binding sum}
\label{sec:fibre_connected_sum}
We will briefly introduce the fibre connected sum, which is a method to construct manifolds using embedded submanifolds; for details see \cite[Section~7.4]{Geiges2008}.
Let $M'$ and $M$ be closed oriented manifolds and let $j_0$ and $j_1$ be embeddings of $M'$ into $M$ with disjoint images. Assume that there exists a bundle isomorphism $\Psi$ of the corresponding normal bundles $N_0$ and $N_1$ over $j_1 \circ j_0^{-1}|_{j_0(M')}$ that reverses the fibre orientation. Picking a bundle metric on $N_0$ and choosing the induced metric on $N_1$ turns $\Psi$ into a bundle isometry. We furthermore identify open disjoint neighbourhoods of the $j_i(M')$ with the normal bundles $N_i$.

The \defterm{fibre connected sum} is the quotient manifold
$$
\#_\Psi M := \Big(  M\setminus \big( j_0(M') \cup j_1(M') \big)  \Big)/_\sim,
$$
where $v \in N_0$ with $0 < \| v \| < \epsilon$ is identified with $\frac{\sqrt{\epsilon^2 - \| v \|^2}}{\| v \|}\Psi(v)$.
It is worth noting that the construction can be adapted to work in the symplectic and contact setting if the dimensions of $M$ and $M'$ differ by two, see Section 7.2 in \cite{McDuff1998} and Theorem 7.4.3 in \cite{Geiges2008} for details.

Let $M$ be a (not necessarily connected) smooth $n$-dimensional manifold with open book decomposition
$(\Sigma, \phi)$ whose binding $B$ contains two diffeomorphic components $B_0,B_1$
with diffeomorphic open book decompositions $(\Sigma^\prime, \phi^\prime)$.
Their normal bundles $\nu B_0$ and $\nu B_1$ admit trivializations induced by the
pages of the open book decomposition of $M$. Let $\Psi$
denote the fibre orientation reversing diffeomorphism of
$B \times D^2 \subset B \times \mathbb{C}$ sending $(b,z)$ to $(b,\bar z)$.
Hence, we can perform the fibre connected sum along $B_0$ and $B_1$ with 
respect to the above trivializations of the normal bundles and the map induced by $\Psi$ 
and denote the result by
\[
	\#_{B_0,B_1} M.
\]
We call it the \defterm{binding sum} of $M$ along $B_0$ and $B_1$.

An interesting case occurs when $M$ is disconnected and decomposes into two manifolds 
$M_0$ and $M_1$ with open books admitting diffeomorphic bindings, denoted by $B$ say. 
In this case we denote the binding sum along the receptive copies of $B$ in $M_0$ and $M_1$
by 
\[
	M_0 \#_{B} M_1.
\]
Note that $M_0 \#_{B} M_1$ admits the structure of a fibration over the circle with fibre given by
\[
			(-\Sigma_0) \cup_B \Sigma_1,
\]
where $\Sigma_0$ and $\Sigma_1$ are the pages of the open book for $M_0$ and $M_1$ respectively. 	In the contact setting each fibre defines a \emph{convex} hypersurface, i.e.\
there exists a contact vector field on $(M_0,\xi_0)\#_B (M_1,\xi_1)$ which is
transverse to the fibres. Furthermore for each fibre $(-\Sigma_0) \cup_B \Sigma_1$
the contact vector field is tangent to the contact structure exactly over $B$
(cf.\ Section~\ref{section:circle_fibrations}).

As stated in the introduction, this paper is concerned with the question whether the binding sum admits again
the structure of an
open book and how it is related to the original open books. 

%%%%%%%%%%%%%%%%%%%%%%%%%%%%%%%%%%%%%%%%%%%%%%%%%%%%%%%%%%%%%
% 
\section{Nested open books}
\label{section:sub_open_book}
%
%%%%%%%%%%%%%%%%%%%%%%%%%%%%%%%%%%%%%%%%%%%%%%%%%%%%%%%%%%%%%
In this section we turn our attention to a special class of submanifolds 
and introduce the notion of a \emph{nested open book}, i.e.\ a
submanifold carrying an open book structure compatible with the open book
structure of the ambient manifold.
We also discuss fibre connected sums
in this context.

Let $M$ be an $n$-dimensional manifold supported by an open 
book decomposition $(B,\pi)$. Let $M'\subset M$ be a 
$k$-dimensional submanifold which on its part is supported by an open book
decomposition $(B',\pi')$ such that
\[
					\pi|_{M'\setminus B'} = \pi'.
\]
Note that $B'$ necessarily defines a $(k-2)$-dimensional submanifold in $B$.
We will always assume
that $M'$ intersects the binding $B$ transversely.
We refer to $M'$, as well as to $(B',\pi')$, as a \defterm{nested open book} of $(B,\pi)$. 

Let $(\Sigma, \phi)$ be an abstract open book and $\Sigma' \subset \Sigma$ a properly embedded
submanifold intersecting $\partial\Sigma$ exactly in its boundary $\partial\Sigma'$ with the intersection being transverse.
We call $(\Sigma', \phi|_{\Sigma'})$ an \defterm{abstract nested open book}
if $\Sigma'$ is invariant under the monodromy $\phi$.
The equivalence of the two definitions follows analogously to the
equivalence of abstract and non-abstract open books.
If not indicated otherwise, we will assume the normal bundle of any
nested open book used in the present paper to be trivial.

\begin{example}
Consider a $k$-disc $D^k \subset D^n$ inside an $n$-disc $D^n$ coming from the natural inclusion $\R^k \subset \R^n$.
This realises
$S^{k+1} \cong (D^k, \id)$ as a nested open book of $S^{n+1} \cong (D^n, \id)$.
The case $k=1$ and $n=2$ is depicted in Figure~\ref{fig:ob_sphere_disc}.
For $k = n-2$, the nested $S^{n-1}$ is a \textit{push-off}, as will be defined in Section~\ref{section:push-off}, of the binding of $(D^n, \id)$.
\end{example}

% 
%%%%%%%%%%%%%%%%%%%%%%%%%%%%%%%%%%%
\subsection{Fibre sums along \subopenbook s}
%%%%%%%%%%%%%%%%%%%%%%%%%%%%%%%%%%%
% 
For the remainder of the section, assume the co-dimension of the nested open books to be two.
%The question whether the resulting manifold of a fibre connected sum operation of two
%open books carries an open book structure can be answered positively in the case when
%the sum is performed along nested open books.
We will show that the fibre connected sum operation of two open books along
diffeomorphic nested open books carries an open book structure with page a fibre
connected sum of the original pages along the nested bindings.

Let $M'$ be an $(n-2)$-dimensional manifold supported by an open book $(\pi^\prime, B^\prime)$, and let 
$j_0,j_1 \colon\thinspace M^\prime \hookrightarrow M$ be two disjoint embeddings
defining nested open books of $M$ such that their images admit isomorphic normal
bundles $N_i$.
We denote by $M^\prime_i := j_i(M^\prime)$ the embedded copies of the nested open book $M^\prime$ and by
$B^\prime_i := j_i(B^\prime)$ their respective bindings. Finally let $\pi^\prime_i := \pi' \circ j_i^{-1}$ denote the induced open book fibration on $M_i^\prime \setminus B_i^\prime$ and let $(\Sigma'_i)_\theta$ denote their pages.

Given an orientation reversing bundle isomorphism $\Psi$ of the normal bundles
$\nu M^\prime_i$, we can perform the fibre connected sum $\#_{\Psi} M$.
We only have to ensure that the fibres of the normal bundles of $M'_0$ and $M'_1$ lie within the pages of $(\pi, B)$. In particular, we require the fibres over the \subbinding s to lie within the binding of $M$.
Moreover, we require the isomorphism $\Psi$ of the normal bundle to respect the  open book structure of $M$ (which implies that it is compatible with the \subopenbook \  structures of $M'_0$ and $M'_1$ as well), i.e.\ $\Psi$ to satisfies $\pi \circ \Psi = \pi$.
Now, an open book structure of $\#_\Psi M$ is given as follows.
\begin{lemma}
\label{lemma:sum_sub_ob}
The original fibration $\pi \colon\thinspace M \setminus B \to S^1$ 
descends to a fibration 
$$
	\Pi \colon\thinspace \#_{\Psi} M  \setminus \#_{\Psi|_{\nu B_0^\prime}} B  \to S^1.
$$
In particular, the new binding is given by the fibre connected sum $\#_{\Psi|_{\nu B_0^\prime}} B $ of the binding along the \subbinding s (with respect to the isomorphism of $\nu B'_i \subset TB$ induced by $\Psi$), and the pages of the open book are given by the (relative) fibre sum of the original page along the \subpage s (with respect to the isomorphism of $\nu {\pi'_i}^{-1}(\theta) \subset T\pi^{-1}(\theta)$ induced by $\Psi$), i.e.\ 
$ 
			 \overline{\Pi^{-1}(\theta)}  = 
						 \#_{\Psi|_{\nu (\Sigma_{0}')_\theta}} \Sigma_\theta .
$
\qed
\end{lemma}

In the following we are going to extract the remaining information to express $\#_\Psi M$ in terms of an abstract open book, that is we describe a recipe to find the monodromy.
Let $X$ be a vector field transverse to the interior of the ambient pages, vanishing on the binding, and normalised by $\pi^*d\theta(X) = 1$. Recall from Section~\ref{section:prelim_open_books} that the time-$2\pi$ map $\phi$ of the flow of $X$ yields the monodromy of the ambient open book.
Furthermore, if we assume that $X$ is tangent to the submanifolds $M^\prime_i$, we obtain abstract \subopenbook \ descriptions $(\Sigma_i, \phi_i)$ of $M^\prime_i$ within the abstract ambient open book $(\Sigma, \phi)$.
Moreover, by adapting the vector field if necessary, we can choose embeddings of the normal bundles of $M^\prime_i$ such that the fibres are preserved under the flow of $X$.
The normal bundles of $M^\prime_0$ and $M^\prime_1$ being isomorphic
translates into the condition that 
the normal bundles $\nu \Sigma^\prime_i$ of the induced (abstract) nested pages
in the ambient (abstract) page $\Sigma$ are $\phi$-equivariantly isomorphic.

For the remaining part of the section we identify $\nu M'_i$ with the quotient 
$$\big( \nu \Sigma'_i \times [0,2\pi] \big) / \sim_{\phi}.$$
Now let $\Psi_0$ be the $\phi$-equivariant fibre-orientation reversing isomorphism of $\nu \Sigma'_i$
induced by the restriction of $\Psi$.
Moreover we define 
\[
	\Psi_t := \Psi|_{\nu \Sigma'_0 \times \{t\}}.
\]
Note that each $\Psi_t$ is isotopic to $\Psi_0$, the whole family $\{ \Psi_t \}_t$ however defines an (\emph{a priori}) non-trivial loop of maps $\nu \Sigma'_0 \to \nu \Sigma'_1$ based at $\Psi_0$.
By choosing suitable bundle metrics,
this loop yields an (\emph{a priori}) non-trivial loop $\{\twistmap_t\}_t$ of maps $\Sigma' \to S^1$ based at the identity via 
\[
			\twistmap_t(x)\cdot \Psi_0(q) := \Psi_t(q),
\]
for $x \in \Sigma'$ and $q \neq 0$ a non-trivial point in the normal-fibre over $x$.
With this in hand we can define a monodromy-like map of $\nu \Sigma'_1$ which is the identity in a neighbourhood of the zero section and outside the unit-disc bundle by
\[
	\twistmap(q) := \twistmap_{\boldsymbol{r}(x)}\cdot q,
\]
where $\boldsymbol{r}$ is a radial cut-off function in the fibre which is $1$ on the
zero section and vanishes away from it.
We call it the \defterm{twist map} induced by $\phi$ and $\Psi$.
Given this map we can now give an abstract description of the open book in Lemma~\ref{lemma:sum_sub_ob}. Recall that we already identified the page as the fibre sum of the original page along the \subpage s.

\begin{lemma}
\label{lem:sum_abstract_sub}
Let $\Psi_0$, $\phi$ and $\twistmap$ be the maps described in the above paragraph.
Then the monodromy of the open book in Lemma~\ref{lemma:sum_sub_ob} is given by $\phi \circ \twistmap$,
and the page is $\#_{\Psi_0} \Sigma$.
\qed
\end{lemma}

%%%%%%%%%%%%%%%%%%%%%%%%%%%%%%%%%%%%%%%%%%%%%%%%%%%%%%%%%%%%%
% 
\section{The push-off}
\label{section:push-off}
% 
%%%%%%%%%%%%%%%%%%%%%%%%%%%%%%%%%%%%%%%%%%%%%%%%%%%%%%%%%%%%%

In this section we describe a \emph{push-off} of the binding of an open book
which realises it as a \subopenbook . The push-off construction will enable us
to describe a natural open book structure on the fibre connected sum of two open
books along their diffeomorphic bindings.
We will first describe how the binding is being pushed away from itself and then
introduce a natural framing of the pushed-off copy in Subsection~\ref{section:framings},
which will be equivalent to the canonical \emph{page framing} of the binding.
In Subsection~\ref{section:monodromy} we show that the push-off can be realised
as an \emph{abstract \subopenbook }.

\begin{figure}[htb]
\includegraphics[width=1.\textwidth]{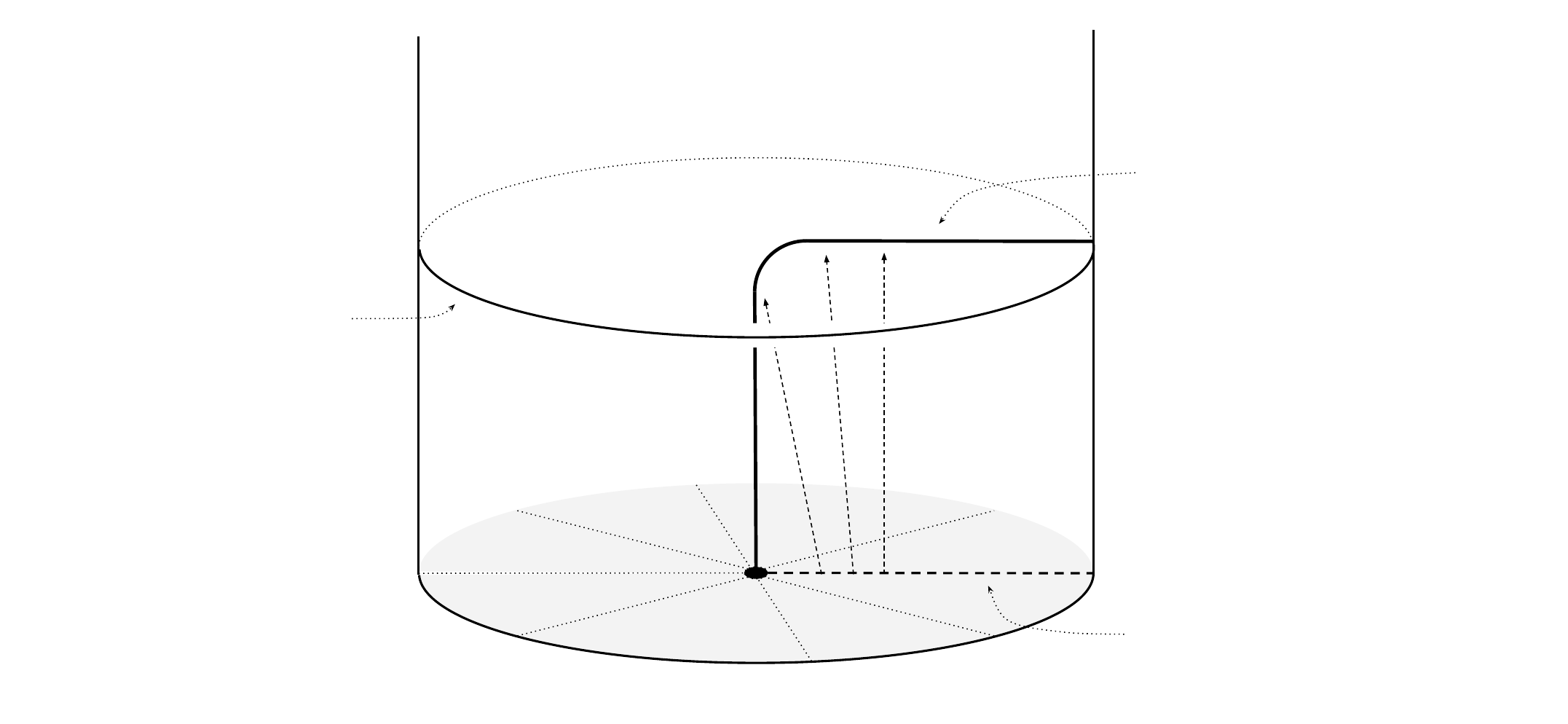}
\put(-305, 92){$r = c$}
\put(-280, 132){$\Sigma_\theta$}
\put(-280, 30){$B$}
\put(-96, 125){$B^+ \cap \Sigma_\theta$}
\put(-100, 18){$\Sigma'_{\theta} \subset B$}
\caption{The page $\Sigma'_\theta$ pushed into $\Sigma_\theta$}
\label{fig:push-off_single_page}
\end{figure}
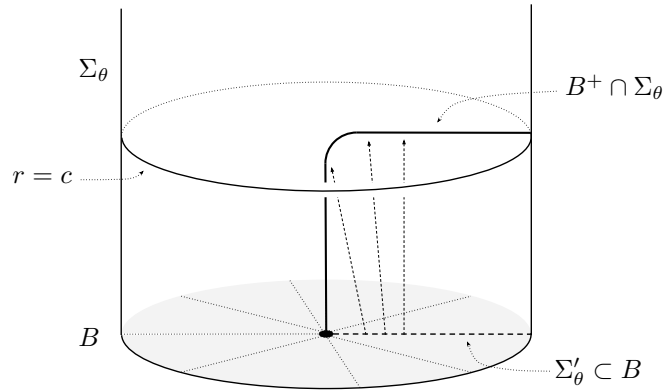

Let $M$ be a manifold with open book decomposition $(\Sigma, \phi)$
and binding $B$ which also admits an open book decomposition
$(\Sigma^\prime, \phi^\prime)$. We denote the fibration maps by
$\pi \colon\thinspace M \setminus B \rightarrow S^1$ and
$\pi' \colon\thinspace B \setminus B' \rightarrow S^1$, respectively.
Our aim is to define a push-off $B^+$ of the binding $B$ in such a way that
each page $\Sigma^\prime_\theta$ of the binding open book is pushed into
$\Sigma_\theta$, the page corresponding to the same angle $\theta$
in the ambient open book.
As all our constructions are local in a neighbourhood of the binding $B$,
we can assume, without loss of generality, that $\phi$ is the identity.

Identify a neighbourhood of the binding $B^\prime \subset B$ of the open book
of the binding $B$ with $B' \times D^2$ with coordinates $(b', r', \theta')$ such that
$(r', \theta')$ are polar coordinates on the $D^2$-factor and $\theta'$ corresponds to the fibration $\pi'$ -- these are standard coordinates for a neighbourhood of a binding of an open book.
We will also use Cartesian coordinates $x'$, $y'$ on the $D^2$-factor.
Analogously, we have coordinates $(b,r,\theta)$ in a neighbourhood
of $B \subset M$ with the corresponding properties.
Combining these, we get two sets of coordinates on $(B' \times D^2) \times D^2 \subset M$:
\[ 
(b', r', \theta', r, \theta) \ \text{ and } \ (b', x', y', x, y).
\] 
First, we will describe the geometric idea of the push-off by considering 
just a single page $\Sigma_\theta$ of the open book
before defining it rigorously afterwards,
see Figure \ref{fig:push-off_single_page}.
The page $\Sigma'$ of the binding open book corresponding to angle $\theta$ is pushed into the page $\Sigma$ of the ambient open book corresponding to the same angle $\theta$.
Restricted to a single page of the binding open book,
the push-off depends on the radial direction $r'$ only
and is invariant in the $B'$-component.
In particular, the boundary of the page $\Sigma'$ stays fixed.
We divide the collar neighbourhood in $\Sigma'$ into four parts
by the collar parameter $r'$.
The outermost one consisting of points in $\Sigma'$ with $r' \leq \epsilon_1$
is mapped to run straight into the $r$-direction of the ambient page $\Sigma$.
The innermost part consisting of points with $r' \geq \epsilon_3$
is translated by a constant $c$ into the $r$-direction.
This translation is extended over the whole of $\Sigma'$.
On the rest of the collar the push-off is an interpolation between these
innermost and outermost parts. This is done such that
points with $\epsilon_1 \leq r' \leq \epsilon_2$ are used to interpolate in
$r$-direction and points with $\epsilon_2 \leq r' \leq \epsilon_3$ in
$r'$-direction.

Let $f, h \colon\thinspace \R^+_0 \rightarrow \R$ be the smooth functions described in Figure~\ref{fig:funtions_def_push-off}, i.e.\ they have the following properties:
 \begin{itemize} 
 \item $f(r) = 0$ for $r \leq \epsilon_1$ and $f(r) = r$ for $r \geq \epsilon_3$,
 \item $f'(r) > 0$ for $r > \epsilon_1$,
 \item $h(0) = 0$, $h'(0) = 1$ and $h(r) \equiv c$ for $r \geq \epsilon_2$,
\item $h(r) > 0$ for $r > 0$ and $h'(r) > 0$ for $0 < r < \epsilon_2$.
\end{itemize}
By choosing $f$ small between $\epsilon_1$ and $\epsilon_2$, the ``curved part"
of the push-off can be realised in arbitrarily small.
Recall that $B$ can be decomposed as $(B' \times D^2) \cup \Sigma'(\phi')$.
Let $g\colon\thinspace B \rightarrow B \times D^2 \subset M$ be the embedding defined by
$$
g(b) =
\begin{cases}
\big( (b', f(r') \cdot e^{i\theta'}), h(r') \cdot  e^{i\theta'} \big) & \text{for } b = (b', r'e^{i\theta'}) \in B' \times D^2 \\
\big( [x', \theta'] , c \cdot  e^i{\theta'} \big) & \text{for } b = [x', \theta'] \in \Sigma'(\phi').
\end{cases}
$$

Observe that $g$ is well-defined and a smooth embedding.

\begin{definition}
\label{def:push_off}
We define the \defterm{push-off $B^+$ of $B$} as the image of the embedding $g$ defined above, i.e. we define
\[
		B^+ := g(B).
\]
Observe that we can easily obtain an isotopy between the binding $B$
and the push-off $B^+$ by parametrising $f$ and $h$.
\end{definition}

\begin{figure}[htb]
\centering
\def\svgwidth{0.8\columnwidth}
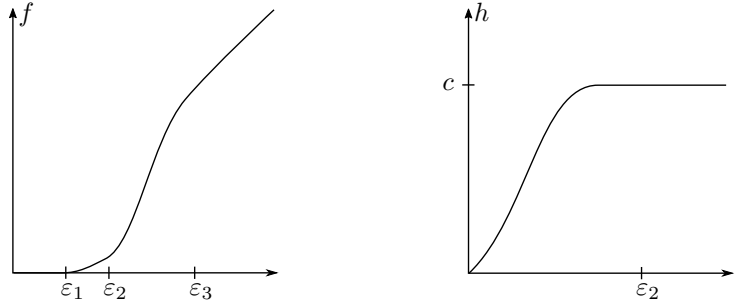
\caption{The functions $f$ and $h$}
\label{fig:funtions_def_push-off}
\end{figure}

\begin{remark}
We call the submanifold $B^+$ a \emph{push-off} of $B$ although it is not
a push-off in the usual sense since $B$ and $B^+$ are not disjoint. 
However, it generalizes the notion of a push-off 
of a \textit{transverse} knot in a contact manifold.
Note that $B^+ \cap \Sigma_\theta \cap \{ r = c \}$ is a copy of the
interior of the page $\Sigma'_\theta$ of the binding and
$B^+ \cap \Sigma_\theta \cap \{ r = r_0\} \cong B' \times \{ r_0 \}$ with $r_0 < c$.
\end{remark}

%%%%%%%%%%%%%%%%%%%%%%%%%%%%%%%%%%%%%%%%%%%%%%%%%%%%%%%%%%%%%
% 
\subsection{Framings}
\label{section:framings}
% 
%%%%%%%%%%%%%%%%%%%%%%%%%%%%%%%%%%%%%%%%%%%%%%%%%%%%%%%%%%%%%

The fibre connected sum explained in
Section~\ref{sec:fibre_connected_sum} requires the submanifolds to have
isomorphic normal bundles and explicitly uses a given bundle isomorphism.
The binding of an open book has trivial normal bundle. Hence it is sufficient to
specify a framing,
i.e.\ a trivialisation of its normal bundle,
to be able to perform a fibre connected sum along the binding.
Note that because the codimension of the binding is two, this can be done
by specifying a push-off, or equivalently a
non-zero vector field along the submanifold that is nowhere tangent, by considering
the normal bundle as a complex line bundle.

A natural framing of the binding $B \subset M$ of an open book is the
\defterm{page framing} obtained by pushing $B$ into one fixed page of the
open book. We denote the page framing given by $\partial_x$ by $F_0$, i.e.
\[
			F_0 := \partial_x.
\]
Next we are going to define a framing for the push-off $B^+$. Let $\widetilde{u} \colon\thinspace M \rightarrow \R$ be a smooth function
such that
\begin{itemize}
\item $\widetilde{u} \equiv 0$ near $B$
and on $B' \times \{r' \leq \epsilon \} \times D^2_{c-\epsilon}$,
\item $\widetilde{u} \equiv 1$ on
$B' \times \{ r' \geq \epsilon \} \times \{ r = c \}$
and outside $\{ r \leq c + \epsilon\}$,
\item $\widetilde{u}$ is monotone in $r'$- and $r$-direction.
\end{itemize}
With this in hand we define a framing of the push-off $B^+$ by
$$
\label{def:F_1}
F_1 := -(1-\widetilde{u}) \partial_{x'}
- \widetilde{u} \cdot (\sin^2 \theta \partial_{x'} - \cos \theta \partial_r).
$$
One easily checks that this is indeed nowhere tangent to $B^+$.
The push-off $B^+$ with the framing $F_1$ is in fact equivalent
to the binding $B$ with its natural page framing $F_0$.

\begin{lemma}
\label{lemma:push-off_isotopic_to_binding}
The framed submanifolds $(B, F_0)$ and $(B^+, F_1)$ are isotopic.
\end{lemma}

\begin{proof}
This is a direct calculation, see \cite{KoelschDurst} for details.
\end{proof}

%%%%%%%%%%%%%%%%%%%%%%%%%%%%%%%%%%%%%%%%%%%%%%%%%%%%%%%%%%%%%
% 
\subsection{The push-off as an abstract \subopenbook }
\label{section:monodromy}
% 
%%%%%%%%%%%%%%%%%%%%%%%%%%%%%%%%%%%%%%%%%%%%%%%%%%%%%%%%%%%%%

The push-off $B^+$ is clearly an embedded \subopenbook \ of $M = (B, \pi)$.
The aim of this section is to obtain a description of the push-off as
an \emph{abstract} \subopenbook , and ultimately as a \emph{framed} abstract \subopenbook \ 
by altering the monodromy of the abstract open
book $(\Sigma, \phi)$.

Identify a neighbourhood of $B' \subset B$ with $B' \times D^2$ as above,
i.e.\ the pages are defined by the angular coordinate $\theta'$.
Also denoting the coordinate on $S^1$ by $\theta'$, we can define a
non-vanishing $1$-form on $B\setminus B'$ by the pull--back of $d\theta'$
under the fibration map $\pi' \colon\thinspace B\setminus B' \rightarrow S^1$.
With the help of this $1$-form we can extend the vector field
$\partial_{\theta'}$ to a vector field $X'$ on $B$ by
prescribing the condition $(\pi')^* d\theta' (X') = 1$.
The vector field $X'$ can furthermore be extended trivially to a neighbourhood
$B\times D^2$ of $B$ in $M$.
Likewise, 
we obtain an abstract open book description of $M$ by
regarding the time-$2\pi$ map of a suitable vector field on $M \setminus B$.
Let $X$ denote the vector field that recovers the abstract open book
$(\Sigma, \phi)$.
Let $u \colon\thinspace \R^+_0 \rightarrow \R$ be the smooth function depicted
in Figure~\ref{fig:function_u} with $c$ as in the definition
of the push-off $B^+$.
Then $\widetilde{X} := X + u(r)X'$ on $M\setminus B$ defines a vector field on
$M\setminus B$.

\begin{figure}[htb]
\centering
\def\svgwidth{0.5\columnwidth}
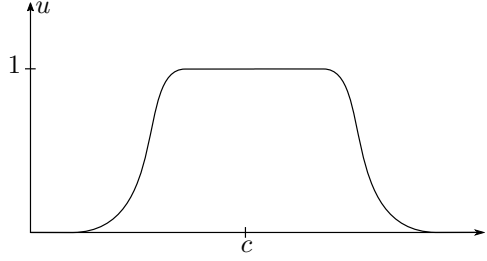
\caption{The function $u$}
\label{fig:function_u}
\end{figure}

We claim that $\widetilde{X}$ realises the push-off $B^+$ as an abstract \subopenbook \ of an abstract open book description of $M$.
Observe that, as $X'$ is tangent to the pages of $(B, \pi)$, the condition
$\pi^* d\theta (\widetilde{X}) = 1$ is satisfied and that $\widetilde{X}$ and
$X$ coincide near the binding $B$. Thus, the vector field $\widetilde{X}$ does
indeed yield an abstract open book description of $M$.
Furthermore, the vector field $\widetilde{X}$ is tangent to the push-off $B^+$,
which means that it realises $B^+$ as an abstract \subopenbook \ of the abstract
ambient open book.
The monodromy is given by the time-$2\pi$ flow of $\widetilde{X}$.
However, we want to give a description that better encodes the change of the
monodromy $\phi$ of the ambient open book we started with in terms of the
monodromy of the binding.

We denote the flow of $X'$ on $B$ by $\phi'_t$ and use it
to define diffeomorphisms $\submonodromy_t$ of a neighbourhood
$B \times D^2$ of $B$ in $M$:
$$
\submonodromy_t (b, r, \theta) = (\phi'_t (b), r, \theta).
$$

\begin{definition}
\label{def:psi}
Define a diffeomorphism $\submonodromy\colon\thinspace M\rightarrow M$ via
$\submonodromy := \submonodromy_{ 2\pi u(r) }$ and refer to it as \defterm{Chinese burn}
along $B$.
By abuse of notation, its restriction to a single page $\Sigma_\theta$ is also
denoted by $\submonodromy$.
\end{definition}

\begin{figure}[htb]
\includegraphics[width=1.\textwidth]{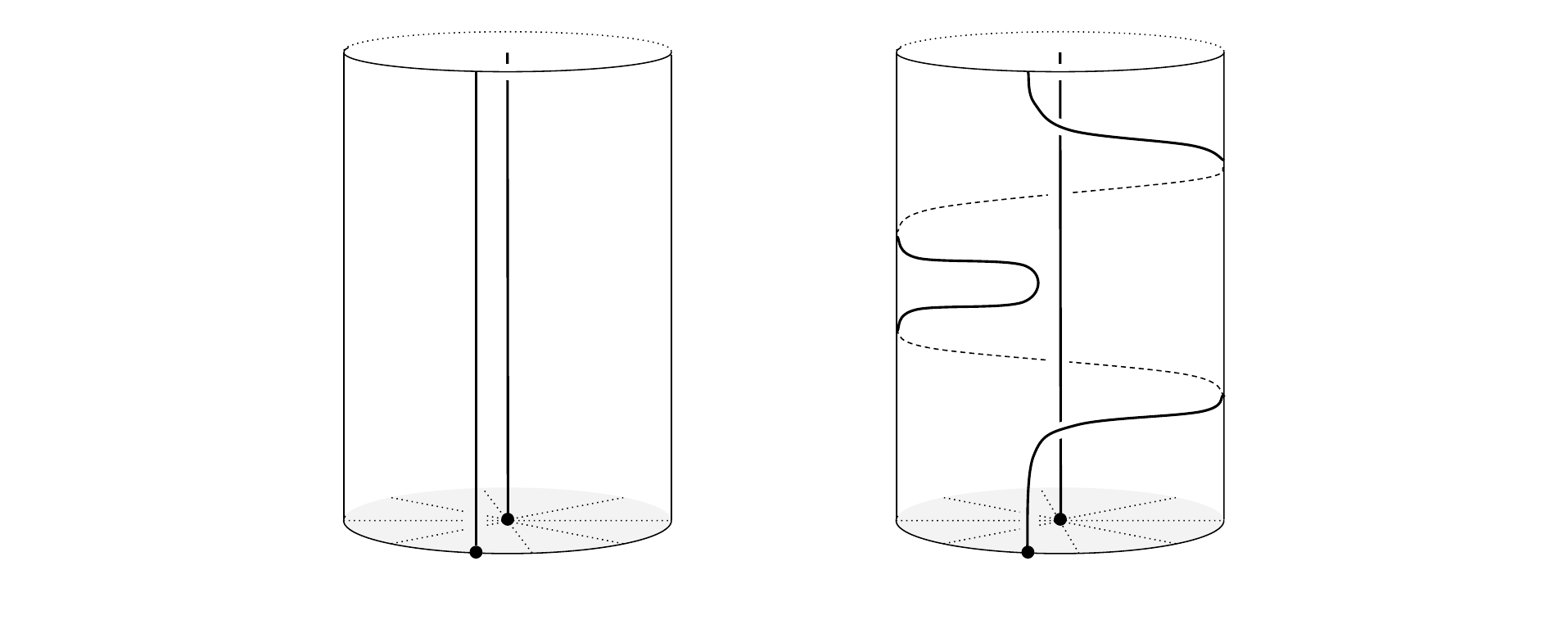}
% \put(-335, 100){$\{ x',r : x' \in B', r\geq 0\}$}
% \put(-335, 60){$x \times \{ r \geq 0\}$}
\put(-187, 70){$\stackrel{\submonodromy}{\longrightarrow}$}
\caption{A Chinese burn along a boundary component.}
\label{fig:submonodromy}
\end{figure}

Observe that the monodromy of the abstract open book obtained from the vector
field $\widetilde{X}$ is $\phi \circ \submonodromy$, i.e.\ the push-off $B^+$ yields an
an abstract \subopenbook \ of $(\Sigma, \phi \circ \submonodromy)$. We thus proved the following statement.

\begin{lemma}
\label{lemma:abstract_sub}
The push-off $B^+$ induces an abstract \subopenbook \ of $(\Sigma, \phi \circ \submonodromy)$
with page diffeomorphic to $\Sigma'$
(more concretely, the page is $g|_{\Sigma'_0}(\Sigma'_0) \cong \Sigma'$),
where $\submonodromy$ is a Chinese burn along $B$.
\end{lemma}

We constructed the push-off $B^+$ inside the manifold $M(\Sigma, \phi)$
and equipped it with a natural framing $F_1$ corresponding to the page framing.
In particular, the push-off is a \emph{framed} \subopenbook ,
i.e.\ a \subopenbook \ with a specified framing.
The previous lemma shows that the push-off also defines an \emph{abstract} \subopenbook \ of
$(\Sigma, \phi \circ \submonodromy)$.
However, the framing $F_1$ does \emph{a priori} not
give a framing in the abstract setting since it is not invariant under the monodromy.
We call an abstract \subopenbook \ with a framing which is invariant under the monodromy
a \defterm{framed abstract \subopenbook }.

\begin{remark}
\label{rem:framedsob}
Given two diffeomorphic framed \subopenbook s with pages $\Sigma'_0,\Sigma'_1  \subset (\Sigma, \phi)$ and 
isomorphic normal bundles,
it is easy to obtain an open book structure of their fibre connected sum.
The new page is
$$\widetilde{\Sigma} := \big( \Sigma \setminus (\nu\Sigma'_0 \cup \nu\Sigma'_1)\big) / \sim$$
with the identification induced by the given framings,
and the old monodromy $\phi$ restricts to the new monodromy.
\end{remark}

A natural framing of the push-off in the abstract setting is the following:
We define the \defterm{constant framing $F_2$} as
$$
F_2 \colon\thinspace 
\widetilde{u} \partial_r - (1-\widetilde{u}) \partial_{x'},
$$
where $\widetilde{u}$ is the restriction of the function
$\widetilde{u} \colon\thinspace M \rightarrow \R$
defined in Section~\ref{section:framings} to a page $\Sigma$.

To realise the push-off as a framed abstract \subopenbook \ with framing $F_2$,
we have to alter the monodromy of the ambient abstract open book. We will change
the monodromy by a certain diffeomorphism of the page fixing the push-off,
the so-called \emph{twist map}.
Let $\sigma \colon\thinspace \R \rightarrow \R$ be a cut-off function with
$\sigma(0) = 0$ and $\sigma(\epsilon) = 1$ and
$$\tau^{\sigma(r)}_\pm \colon\thinspace D^2 \rightarrow D^2$$
a smoothened Dehn twist of the disc.
That is, a diffeomorphism with the qualitative behaviour of
$$
(s, \theta) \mapsto (s, \theta \pm 2\pi \sigma(r) (1-s) ),
$$ smoothened near the boundary and the origin,
such that the origin is an isolated fixed point
and a neighbourhood of the boundary is fixed. This can be achieved by
constructing it as the flow of an appropriate vector field.
Recall that, by construction, the intersection of the push-off of the
binding $B$ with a single page $\Sigma$ of the ambient open book
is a copy of a page $\Sigma'$ of the open book of the binding.
In particular, we can identify a tubular neighbourhood of
$B^+ \cap \Sigma_0$ in $\Sigma_0$
with $\Sigma' \times D^2 \subset \Sigma$.
Observe that the $r$-coordinate can be regarded as a collar parameter
on $\Sigma'$.
We define a diffeomorphism of a neighbourhood of the collar by
\begin{eqnarray*}
\partial\Sigma' \times [0,\epsilon] \times D^2 & \rightarrow &
\partial\Sigma' \times [0,\epsilon] \times D^2 \\
(b', r, p) & \mapsto & (b', r, \tau^{\sigma(r)}_-(p) ). \\
\end{eqnarray*}
We can extend this to a diffeomorphism 
$\Sigma' \times D^2 \rightarrow \Sigma' \times D^2$
of the whole tubular neighbourhood of
$\Sigma'$ via $\id_{\Sigma'} \times \tau^1_-$.
Furthermore, this map can be extended to a self-diffeomorphism
$\twistmap \colon\thinspace \Sigma \rightarrow \Sigma$
of the page $\Sigma$ via the identity.

\begin{definition}
\label{def:twist_map}
The diffeomorphism $\twistmap \colon\thinspace \Sigma \rightarrow \Sigma$ is
called \defterm{twist map}.
\end{definition}

The twist map $\twistmap$ is isotopic to the identity, so we have
$M_{(\Sigma, \phi \circ \submonodromy)} \cong
M_{(\Sigma, \phi \circ \submonodromy \circ \twistmap)}$.

\begin{lemma}
\label{lem:abstract_framed_sub}
The push-off $B^+$ with its induced framing $F_1$ corresponds to the framed abstract \subopenbook \ of $(\Sigma, \phi \circ \submonodromy \circ \twistmap)$ with page $g|_{\Sigma'_0}(\Sigma'_0) \cong \Sigma'$ framed by the natural framing $F_2$.
\end{lemma}

\begin{proof}
Again, this is a straightforward calculation, see \cite{KoelschDurst} for details.
\end{proof}

%%%%%%%%%%%%%%%%%%%%%%%%%%%%%%%%%%%%%%%%%%%%%%%%%%%%%%%%%%%%%
\section{The main result}
\label{section:theorem}

Combining the results of the previous sections, in particular Lemma~\ref{lemma:push-off_isotopic_to_binding}, which states that the binding is isotopic to its push-off, and Lemma~\ref{lemma:sum_sub_ob}, which yields a natural open book structure on the fibre connected sum along push-offs, 
we are now able to give an explicit open book decomposition for the binding
sum operation and thus
state our main result, or, more concretely, the following theorem.

\begin{theorem}
\label{thm:main1}
Let $M$ be a (not necessarily connected) smooth manifold with open book decomposition
$(\Sigma, \phi)$ whose binding $B$ contains two diffeomorphic components $B_0,B_1$
with diffeomorphic open book decompositions $(\Sigma^\prime, \phi^\prime)$.
Then the fibre connected sum of $M$ performed along $B_0$ and $B_1$
with respect to the page framings admits an open book decomposition naturally adapted to the construction.
The new page is the fibre connected sum of the page $\Sigma$ along the nested pages
$\Sigma'_i$ induced by the push-offs of the bindings components $B_i$.
The new binding is 
given by a fibre connected sum of $B_0$ and $B_1$ along their respective bindings $B^\prime = \partial \Sigma$.
The monodromy remains unchanged outside of a neighbourhood of the original
binding components $B_i$, and over the remaining part it 
restricts to $\submonodromy \circ \twistmap$, where $\submonodromy$ is a
\emph{Chinese burn} along $B_i$ (see Definition~\ref{def:psi}) and $\twistmap$ the \emph{twist map}
(see Definition~\ref{def:twist_map}).\qed
\end{theorem}

It is worth noting that the
new page can be obtained from the old one by two consecutive generalised
$1$-handle attachments of type $\Sigma'$, where
a generalized $1$-handle in the sense of \cite{Klukasa} is of the form $H_{\Sigma'} = D^1 \times (\Sigma' \times D^1)$.
In case $\Sigma$ can be endowed with an appropriate exact symplectic form,
generalised $1$-handles can be adapted to the contact setting, and naturally extends the symplectic handle
constructions due to Eliashberg \cite{Eliashberg90} and Weinstein \cite{Weinstein91}.

%%%%%%%%%%%%%%%%%%%%%%%%%%%%%%%%%%%%%%%%%%%%%%%%%%%%%%%%%%%%%
%
\section{Application to contact topology}
\label{section:contact}
% 
%%%%%%%%%%%%%%%%%%%%%%%%%%%%%%%%%%%%%%%%%%%%%%%%%%%%%%%%%%%%%

In this section we want to show that the binding sum of contact open books yields
an open book and a contact structure that again fit nicely together.

For an introduction to open books in contact topology we kindly refer the reader to
\cite{Etnyre2004,Geiges2008,Koert2010}. We will always assume contact
structures to be coorientable.

A positive contact structure $\xi$ on an oriented manifold $M$ is \defterm{supported}
by an open book structure $(B, \pi)$ if it can be written as the kernel of a
contact form $\alpha$ inducing a positive contact structure on $B$ and such that
$d\alpha$ induces a positive symplectic structure on the fibres of $\pi$.
Such a contact form $\alpha$ is then called \defterm{adapted} to the open book.
Note that if the binding $B$ of an open book $(B, \pi)$ is already assumed to be
a contact manifold, then a contact form is adapted to the open book if and only if
its Reeb vector field is positively transverse to the fibres of $\pi$
(cf.~\cite[Lemma~2.13]{Koert2010}).

\begin{theorem}
\label{thm:contact}
The binding sum construction can be made compatible with the underlying
contact structures,
i.e.\ the contact structure obtained by the contact fibre connected sum along
contactomorphic binding components is supported by the natural open book structure
resulting from the sum along the push-offs of the respective binding components.
\end{theorem}

We proceed by introducing some further terminology needed in the remainder of the section.
Let $(M,\xi = \ker \alpha)$ be a contact manifold supported by an open book $(B, \pi)$ and denote the pull-back of $d\theta$ under $\pi\colon\thinspace M\setminus B\rightarrow S^1$ also by $d\theta$.
A vector field $X$ is called \defterm{monodromy vector field} if
\begin{itemize}
\item it is transverse to the pages and satisfies $d\theta(X) = 1$ on $M \setminus B$,
\item the restriction of $\mathcal{L}_Xd\alpha$ to any page vanishes,
\item it equals $\partial_\theta$ on a neighbourhood $B \times D^2 \subset M$ of the binding, where $(r,\theta)$ are polar coordinates on the $D^2$-factor, and the open book fibration is given by the angular coordinate on $D^2$.
\end{itemize}

Given a monodromy field we get an associated abstract contact open book description of $(M,\xi)$,
i.e.\ a triple $(\Sigma, \phi, d\lambda)$ consisting of an exact symplectic page $(\Sigma, d\lambda)$ and a symplectomorphism $\phi$ such that the Liouville vector field $Y$ defined by $i_Y\omega = \lambda$ is
transverse to the boundary $\partial\Sigma$ pointing outwards (and thus induces a contact structure on $\partial\Sigma$),
and in turn an identification of $(M,\xi)$ as a generalised Thurston--Winkelnkemper construction (cf.\ \cite[Section~7.3]{Geiges2008}).
In particular, such a vector field always exists if the open book and contact structure comes from a generalized Thurston--Winkelnkemper construction.
Since any contact manifold supported by an open book can be realised by a generalised Thurston--Winkelnkemper construction (e.g.\ see \cite[Section~7.3]{Geiges2008} or \cite[Theorem~3.1.22]{Max_Thesis}), we can always assume the existence of a compatible monodromy vector field.
An adapted Reeb vector field is not a monodromy vector field as it does not fix the binding point-wise.

Define a \defterm{Lutz pair} $h_1, h_2$ as a pair of smooth functions
$h_1 \colon\thinspace [0,1]\rightarrow \R^+$ and
$h_2 \colon\thinspace [0,1]\rightarrow \R^+_0$ such that
\begin{itemize}
\item $h_1(0) = 1$ and $h_2$ vanishes like $r^2$ at $r = 0$,
\item $h_1'(r) < 0$ and $h_2'(r) \geq 0$ for $r>0$,
\item all derivates of $h_1$ vanish at $r = 0$.
\end{itemize}

We will now define nested open books in the contact world.

\begin{definition}
\label{def:ContactNested}
Let $(\Sigma, d\lambda, \phi)$ be a contact open book and $\Sigma'\subset\Sigma$
a symplectic submanifold with boundary $\partial\Sigma'\subset\partial\Sigma$
such that in a collar neighbourhood $\partial\Sigma\times (-\varepsilon, 0]$ of 
$\partial\Sigma$ given by an outward-pointing Liouville vector field
we have
$\Sigma'\cap\big(\partial\Sigma\times (-\varepsilon, 0]\big)
 = \partial\Sigma'\times (-\varepsilon, 0]$.
Suppose furthermore that $\phi(\Sigma') = \Sigma'$, i.e.\ the monodromy leaves
$\Sigma'$ invariant (not necessarily pointwise). Then $\Sigma'$ is a
\defterm{contact abstract nested open book} of the contact open book
$(\Sigma, d\lambda, \phi)$.
\end{definition}

Note that the fibre connected sum construction can be adapted to the contact
setting as follows (we will only apply it in the case of trivial normal bundles where the existence of the bundle isomorphism $\Phi$ is always granted).

\begin{theorem}[{\cite[Theorem 7.4.3]{Geiges2008}}]
\label{thm:contact_fibre_sum}
Let $(M, \xi)$ and $(M', \xi')$ be contact manifolds of dimension 
$\dimension M' = \dimension M - 2$, where the contact structures $\xi, \xi'$
are assumed to be cooriented; these cooriented contact structures induce orientations of $M$ and $M'$. 
Let $j_0, j_1\colon\thinspace (M', \xi')\rightarrow (M, \xi)$ be disjoint contact embeddings that respect the coorientations, and such that there exists a fibre-orientation-reversing bundle isomorphism $\Phi\colon\thinspace N_0\rightarrow N_1$
of the normal bundles of $j_0(M')$ and $j_1(M')$.
Then the fibre connected sum $\#_\Phi M$ admits a contact structure that coincides
with $\xi$ outside tubular neighbourhoods of the submanifolds $j_0(M')$ and $j_1(M')$.
\end{theorem}

The binding of a contact open book decomposition is a contact submanifold and hence
admits an open book structure itself. Furthermore, it has trivial normal bundle.
Given two contact open books with contactomorphic bindings, we can thus perform the
contact fibre connected sum along their bindings, and, topologically, also along
the push-offs $B^+_i$ of the bindings. We want to show that this topological
construction can be adapted to the contact scenario.

In fact, the push-off is a contact submanifold contact isotopic to the binding. Thus,
we can form the contact fibre connected sum along the push-off rather than along
the binding itself when performing the contact binding sum.

\begin{proposition}
\label{proposition:contact_push-off}
The push-off $B^+$ of the binding $B$ of a contact open book $(B, \pi)$ is a
contact submanifold contact isotopic to the binding.
\end{proposition}

\begin{proof}
We first show that the push-off $B^+$ of the binding $B$ is a contact submanifold of
$(M, \xi)$.
The binding $B$ is a contact submanifold, so in a neighbourhood $B\times D^2$
(as described in Section~\ref{section:push-off}) we can assume our contact form
$\alpha$ to be
$$
\alpha = h_1\alpha_B + h_2 d\theta,
$$
where $\alpha_B$ is a contact form on $B$ and $(h_1(r), h_2(r))$ a Lutz pair.
To prove that the push-off $B^+$, which arises as the image of the embedding
$g\colon\thinspace B \rightarrow M$
(see Definition~\ref{def:push_off}),
is a contact submanifold of $M$, we have to show that $g^*\alpha$ is a contact form on $B$.
We will first verify this condition away from the binding $B'$ of $B$.
Here, $g$ sends an element $[x',\theta']$ in the mapping torus part of the open
book of $B$ to $([x',\theta'], c, \theta')\in B\times D^2$ for constant $c>0$.
So we have
$$
g^*\alpha = h_1(c)\alpha_B + h_2(c)d\theta'
$$
and $dg^*\alpha = h_1(c) d\alpha_B$.
Hence,
$$
g^*\alpha \wedge (dg^*\alpha)^{n-1} =
\big(h_1(c)\big)^n\alpha_B\wedge (d\alpha_B)^{n-1}
+ \big(h_1(c)\big)^{n-1} h_2(c) d\theta'\wedge (d\alpha_B)^{n-1}.
$$
Observe that $(h_1(c))^{n-1} h_2(c) > 0$. Furthermore,
since $d\theta'(R_{\alpha_B}) > 0$, the forms
$\alpha_B\wedge (d\alpha_B)^{n-1}$ and
$\big(h_1(c)\big)^{n-1} h_2(c) d\theta'\wedge (d\alpha_B)^{n-1}$
induce the same orientation,
i.e.\ $g^*\alpha$ is indeed a contact form away from the binding.

Now we inspect the situation near the binding $B'$ of $B$, i.e.\ we can work
in a neighbourhood $B' \times D^2 \times D^2$ (as described in
Section~\ref{section:push-off}).
The contact form $\alpha_B$ of the binding can be assumed to be of the form
$$
\alpha_B = g_1 \alpha_{B'} + g_2 d\theta',
$$
where $\alpha_{B'}$ is a contact form on $B'$
and $\big(g_1(r'), g_2(r')\big)$ a Lutz pair.
Thus, we have
$$
\alpha = h_1 g_1 \alpha_{B'} + h_1 g_2 d\theta' + h_2 d\theta.
$$
Recall that the defining embedding $g$ for the push-off is given by
$$
g(b', r', \theta') = \big(b', f(r'), \theta', h(r'), \theta'\big)
$$
in this neighbourhood.

We compute
$$
g^*\alpha = \lambda \alpha_{B'} + \mu d\theta'
$$
with
$$
\lambda (r') = (h_1 \circ h) (g_1 \circ f)(r')
$$
and
$$
\mu (r') = \big((h_1 \circ h) (g_2 \circ f)  + h_2 \circ h\big)(r').
$$
So we have
$$
g^*\alpha = \lambda' \alpha_{B'} + \lambda d\alpha_{B'} + \mu' d\theta',
$$
$$
(dg^*\alpha)^{n-1} = (n-1) \lambda^{n-2} (d\alpha_{B'})^{n-2}
\wedge (\lambda' dr' \wedge \alpha_{B'} + \mu' dr' \wedge d\theta')
$$
and
$$
(g^*\alpha) \wedge (dg^*\alpha)^{n-1}
= \frac{1}{r'} (n-1) \lambda^{n-2} (\lambda\mu' - \lambda'\mu)
\big(\alpha_{B'} \wedge (d\alpha_{B'})^{n-2} \wedge r'dr' \wedge d\theta'\big).
$$
It remains to show that the term $\lambda\mu' - \lambda'\mu$ is positive.
A calculation shows
\begin{eqnarray*}
\lambda\mu' - \lambda'\mu & = &
\underbrace{
(h_1 \circ h)^2 \big((g_1 \circ f) (g_2 \circ f)' - (g_1 \circ f)' (g_2 \circ f)
\big)}_{=: A}\\
& + &
\underbrace{
(h_1 \circ h)\big((h_2 \circ h)' (g_1 \circ f) - (g_1 \circ f)' (h_2 \circ h)
\big)}_{=: B}\\
& + &
\underbrace{
\big(-(h_1 \circ h)' (g_1 \circ f) (h_2 \circ h)\big)}_{=: C}.\\
\end{eqnarray*}
Observe that all three summands $A$, $B$, and $C$ are non-negative. It thus
suffices to show that at least one of them is positive.
Assume $C = 0$. Then either $r' = 0$ or $h' = 0$.
We will first deal with the case $h' = 0$. This happens exactly where $h\equiv c$.
But on this set, both $f$ and its derivative $f'$ are positive and
therefore
$$
A = \big(h_1 \circ h(c)\big)^2 f' \big((g_1 g_2' - g_1'g_2) \circ f\big) > 0,
$$
because $(g_1, g_2)$ is a Lutz pair, i.e.\ in particular
$(g_1 g_2' - g_1'g_2) > 0$.
Now consider $r' = 0$.
For small $r'$ we have $f\equiv 0$ and thus
$g_1 \circ f \equiv 1$ and $g_2 \circ f \equiv 0$.
Then $\lambda\mu' - \lambda'\mu$ reduces to
$$
(h_1 \circ h) (h_2 \circ h)' - (h_1 \circ h)' (h_2 \circ h)
= h' \big((h_1 h_2' - h_1' h_2) \circ h\big).
$$
As $h'$ is positive for small $r'$ and $(h_1, h_2)$ is a Lutz pair, this is
positive.
Hence,
$(g^*\alpha) \wedge (dg^*\alpha)^{n-1}$ is a volume form and so $(g^*\alpha)$ a
contact form, i.e.\ the push-off $B^+$ is indeed a contact submanifold.

Observe that by varying $c$ and the $\varepsilon_i$ in the definition of the push-off, we get parametrised versions of the functions $f$ and $h$ yielding a homotopy from $f$ to the identity and from $h$ to the constant zero function, which in turn gives a topological isotopy betweeen the push-off and the binding.
A similar calculation to the one above then
 shows that the push-off $B^+$ is furthermore contact isotopic
to the binding $B$.
\end{proof}

%%%%%%%%%%%%%%%%%%%%%%%%%%%%%%%%%%%%%%%%%%%%%%%%%%%%%%%%%%%%%

A first step of showing that the topological binding sum construction along the
push-off can be made compatible with the underlying contact structures is to
show that the push-off can be realised as a suitable abstract contact nested open book.
To this end, we first show that there exists a monodromy vector field
tangent to the push-off.

\begin{proposition}
\label{proposition:tangent_monodromy}
There exists a monodromy vector field (in the contact geometric sense defined above) tangent to the push-off.
\end{proposition}

\begin{proof}
Without loss of generality, we can assume that the ambient contact manifold $(M, \ker\alpha)$
with contact open book $(B, \pi)$
arises from an abstract contact open book $(\Sigma, d\lambda, \phi)$ by a generalised
Thurston--Winkelnkemper construction.
The push-off $B^+$ of the binding $B$ is contained in a
trivial neighbourhood $B\times D^2$, on which the monodromy $\phi$ restricts
to the identity and the contact form $\alpha$ is given by
$\alpha = h_1 \alpha_B + h_2 d\theta$, where $(h_1, h_2)$ is a Lutz pair and
$\alpha_B$ a contact form on $B$ with the induced contact structure.
Hence, we can furthermore assume that $\phi = \id_\Sigma$, which implies that
$\partial_\theta$ is a monodromy vector field.

By construction, the intersection of the push-off $B^+$ with the fibres of $\pi$
is a cylinder over the binding $B'$ of the compatible open book decomposition with
page $\Sigma'$ used in the construction of the push-off, i.e.\
for some small constant $k > 0$, we have
$$
B^+ \cap \pi^{-1}(\theta) \cap \{r\leq\varepsilon\} = B'\times (0, k]
\subset B\times D^2.
$$
We will now work in $M\setminus (B\times D^2_\varepsilon) \cong \Sigma\times S^1$
with $\varepsilon < k$
(i.e.\ we will work in a trivial mapping torus).

Observe that $B^+ \cap \big(\Sigma \times \{\theta\}\big)$ is a
codimension $2$ symplectic submanifold with trivial normal bundle of
$(\Sigma \times \{\theta\}, d\alpha|_{\Sigma \times \{\theta\}})$. Also,
the symplectic structure on $\Sigma \times \{\theta\}$ induced by $d\alpha$ is
independent of $\theta$. Thus, we get a fibre-wise symplectic projection
$$
p\colon\thinspace \Sigma\times S^1\rightarrow \Sigma.
$$

This defines a family of symplectic submanifolds
$$
\Sigma'_t := p\big(B^+\cap(\Sigma\times\{t\})\big)
$$
of $\Sigma$ which all coincide near the boundary.

Auroux's version of Banyaga's isotopy extension theorem
(see \cite[Proposition~4]{Auroux1997}, cf.~\cite[Theorem~3.19]{McDuff1998}
and the proof of Corollary~\ref{corollary:abstract_nested} below)
for symplectic submanifolds then yields a symplectic isotopy
$$
\phi_t\colon\thinspace \Sigma \rightarrow \Sigma
$$
with $\phi_t(\Sigma'_0) = \Sigma'_t$
in such a way that it is equal to the identity in a neighbourhood
$U_1$ of $\partial\Sigma$
and outside a bigger neighbourhood $U_2$ of the boundary.

Differentiating $\phi_t$ yields a time-dependent vector field $X_t$ on
$\Sigma$,
which can be assumed to coincide for $t=0$ and $t=2\pi$
(it extends a vector field along the submanifold $\Sigma'$ with that property).
Thus, $X_t$ lifts to a vector field $X$ on $\Sigma \times S^1$ with $d\theta(X) = 1$,
whose projection to each fibre $\Sigma \times \{\theta\}$ is symplectic.
Furthermore, the vector field $X$ is equal to $\partial_\theta$ inside
$U_1 \times S^1$ and outside $U_2 \times S^1$. To simplify notation, we set
$V := (U_2\setminus U_1)\times S^1$.

Now given any monodromy vector field $Y$ which is equal to $\partial_\theta$ on
$U_2\times S^1$, we can replace $Y$ by $X$ over $V$
to get a vector field $\widetilde{Y}$, which is
tangent to the push-off by construction.
We claim that $\widetilde{Y}$ is a monodromy vector field.

Indeed, we have $d\theta(\widetilde{Y}) = 1$ and near the binding $\widetilde{Y}$
equals $\partial_\theta$.
Furthermore, the Lie derivative of $d\alpha$ with respect to $\widetilde{Y}$
coincides with $\mathcal{L}_Yd\alpha$ outside of $V$ and with
$\mathcal{L}_X d\alpha$ on $V$. Hence, as $Y$ is a monodromy vector field and
$X$ is symplectic on pages, the restriction of 
$\mathcal{L}_{\widetilde{Y}}d\alpha$ to any page vanishes, which means that
$\widetilde{Y}$ is a monodromy vector field tangent to the push-off $B^+$.
\end{proof}

\begin{corollary}
\label{corollary:abstract_nested}
The push-off is an abstract nested open book of an abstract open book description
with page $\Sigma$ and monodromy $\phi$, where $\phi$ restricts to
$$
\phi|_{\Sigma'\times D^2} = \phi' \times \id_{D^2}
$$
for a symplectomorphism $\phi'\colon\thinspace\Sigma'\rightarrow\Sigma'$
in a neighbourhood $\Sigma'\times D^2$ of $\Sigma'$ given by the symplectic normal
bundle.
\qed
\end{corollary}

\begin{proof}
The main ingredient in the proof is
Auroux's version of Banyaga's isotopy extension theorem \cite[Proposition~4]{Auroux1997} (or rather its proof) which states that for a familiy $W_t$ ($t\in [0,1]$) of symplectic submanifolds of a symplectic manifold $X$, there exist symplectomorphisms $\phi_t\colon\thinspace X\rightarrow X$ (depending continuously on $t$) such that $\phi_0 = \id$ and $\phi_t(W_0) = W_t$.

By the symplectic neighbourhood theorem (see \cite[Theorem~3.30]{McDuff1998})
a neighbourhood of $\Sigma'$ in $\Sigma$ can be written as
$\Sigma' \times D^2$ with split symplectic form.
A smooth family of symplectic submanifolds can be assumed to arise as the image
of an isotopy.
The first step in proving Auroux's theorem is
to extend this isotopy to an open neighbourhood. In our case, the symplectic
neighbourhood theorem allows us to do this in a trivial way.
Applying Auroux's construction to this trivial
extension yields an isotopy which is invariant in fibre direction, i.e.\ the
time-$2\pi$ map of the resulting isotopy restricts to a map of the form
$\phi|_{\Sigma'\times D^2} = \phi' \times \id_{D^2}$ on the neighbourhood of $\Sigma'$.
Hence, we have the following corollary.
\end{proof}

\begin{remark}
Observe that by requiring the monodromy to be trivial in fibre direction,
the resulting abstract nested open book is also a framed nested abstract open book
in the sense of Section~\ref{section:monodromy}.
The monodromy does not correspond to the Chinese
burn $\Psi$ but to the concatenation $\mathcal{D}\circ\Psi$ with the twist map
$\mathcal{D}$, which was used to turn a nested open book into a framed nested
open book in the topological setting and exactly ensured triviality in fibre direction.
\end{remark}

Having described the push-off as an abstract nested open book, it is natural
to perform a fibre sum construction of the ambient abstract open books.
This will then yield a contact structure adapted to the resulting natural open book
decomposition. However, it is a priori unclear whether this contact structure
is indeed the contact structure resulting from the contact fibre connected sum.

We will first show that the symplectic fibre connected sum
(cf.\ \cite[Section~7.2]{McDuff1998})
 of exact symplectic
manifolds is exact under suitable conditions.

\begin{remark}
\label{rem:RelSum}
Note that the symplectic fibre sum can be adapted to work in a relative setting.
Let $W$ be a symplectic manifold with contact type boundary and let $X$ be a
codimension $2$ symplectic submanifold with trivial normal bundle and contact
type boundary which is a contact submanifold of $\partial X$.
Let $\partial W\times (-\epsilon,0]$ be a collar neighbourhood of $\partial W$
given by a Liouville field transverse to $\partial W$ and suppose that
$$
X\cap \big(\partial W\times (-\epsilon,0]\big) = \partial X\times (-\epsilon,0].
$$
The symplectic fibre sum can be performed in this setting with the effect on the
boundary being a contact fibre connected sum.
\end{remark}

The following technical lemma will be useful to interpolate between Liouville
forms which agree on a symplectic submanifold.

\begin{lemma}
\label{lemma:exact}
Let $M$ be a manifold and let $\lambda_0$ and $\lambda_1$ be two $1$-forms on
$M \times D^2$ such that $d\lambda_0 = d\lambda_1$ and
$\lambda_0|_{T(M\times\{0\})} = \lambda_1|_{T(M\times\{0\})}$.
Then $\lambda_1 - \lambda_0$ is exact.\qed
\end{lemma}

\begin{proposition}
Let $(W_i, \omega_i = d\lambda_i)$ ($i=0, 1$), be exact symplectic manifolds
with symplectomorphic submanifolds $X_i \subset W_i$ of codimension $2$
whose normal bundles are trivial.
Assume furthermore that the restriction $\lambda|_{TX_i}$ of the Liouville
forms to these submanifolds coincide through the above identification.
Then the symplectic fibre sum of the $W_i$ along the $X'_i$ is again
exact symplectic.
\end{proposition}

\begin{proof}
We will drop the indices in the first part of the proof and work in $W_0$
and $W_1$ separately.
The submanifold $X \subset W$ is symplectic with symplectic form
$$
\omega' := \omega|_{TX} = (d\lambda)|_{TX} = d(\lambda|_{TX}),
$$
i.e.\ $X$ is exact symplectic and a Liouville form is given by
$\lambda' := \lambda|_{TX}$.
As $X$ is of codimension $2$ and has trivial normal bundle,
the symplectic neighbourhood theorem (see \cite[Theorem~3.30]{McDuff1998})
allows us to write a neighbourhood of $X$ in $W$ as
$X \times D^2$ with symplectic form given as
$\omega = \omega' + sds\wedge d\vartheta$, where $s$ and $\vartheta$ are
polar coordinates on the $D^2$-factor.
In these coordinates one (local) primitive of $\omega$ is given by
$\lambda' + 1/2 s^2 d\vartheta$.
Hence, by Lemma~\ref{lemma:exact}, we have
$$
\lambda|_{X\times D^2} = \lambda' + \frac{1}{2}s^2 d\vartheta + dh
$$
for an appropriate function $h$.

Now by assumption, the restriction of the Liouville forms to the
symplectic submanifolds $X_i$
agree, so in the coordinates adapted to the symplectic normal bundle
as above they are
$\lambda_i = \lambda' + {1}/{2}s^2 d\vartheta + dh_i$.

It follows that the symplectic fibre sum (cf.\ \cite[Section~7.2]{McDuff1998})
along the symplectic submanifolds $X_i$ is again exact symplectic.
The Liouville form can be chosen to coincide with the original ones outside the
area of identification and is given by
$$
\lambda' + \frac{1}{2}s^2 d\vartheta + d\big((1-g)h_0 + gh_1\big)
$$
on the annulus of identification.
Here $g$ is a function on the annulus equal to
$1$ near one boundary and equal to $0$ near the other boundary component.
\end{proof}

We can now apply the proposition to the setting of abstract open books, which
gives a contact version of the topological fibre sum of nested open books described
in Remark~\ref{rem:framedsob}.

\begin{corollary}
Let $\Sigma'_i \subset (\Sigma_i, d\lambda_i, \phi_i)$, $i=0,1$, be two contact
abstract open books with trivial normal bundle.
Let $\psi\colon\thinspace \nu(\Sigma'_0)\rightarrow \nu(\Sigma'_1)$
be a symplectomorphism of neighbourhoods with $\psi(\Sigma'_0) = \Sigma'_1$
satisfying $\psi\circ\phi_0 = \phi_1\circ\psi$
and $(\psi^*\lambda_1)|_{T\Sigma'_0} = \lambda_0|_{T\Sigma'_0}$.
Then the fibre connected sum of the $\Sigma_i$ along the $\Sigma'_i$ with respect to
$\psi$ yields again an abstract open book.

In particular, the symplectic fibre sum of two abstract open books  along the
push-offs of their contactomorphic bindings yields again an abstract open book.
Hence,
the topological binding sum along two contactomorphic binding components carries
a contact structure which is adapted to the natural open book structure and
coincides with the original structures outside a neighbourhood of the push-offs
of the respective binding components.
\end{corollary}

\begin{proof}
We have to show that the symplectic fibre sum is again an exact symplectic
manifold with a Liouville vector field pointing outwards at the boundary
and that the original monodromies give rise to a monodromy on the fibre sum.
The latter is ensured by the condition $\psi\circ\phi_0 = \phi_1\circ\psi$.
Exactness follows almost immediately from the preceding proposition.
Observe that $\Sigma'_i$ being contact abstract nested open books
(cf.\ Definition~\ref{def:ContactNested}) allows us
to perform a relative version of the
symplectic fibre connected sum as described in Remark~\ref{rem:RelSum}.
So by the proposition the
fibre sum yields an exact symplectic manifold with boundary
with Liouville field still pointing outwards.

Note that the description of the push-off as an abstract open book as in
Corollary~\ref{corollary:abstract_nested} fulfils the hypothesis of the first
part of this corollary (the Liouville forms can be assumed to agree as the push-offs
live in trivial neighbourhoods of contactomorphic binding components).
Performing a generalised Thurston--Winkelnkemper construction on the resulting
abstract open book then yields the second part of this corollary.
\end{proof}

\subsection{Naturality of the contact structure}

The preceding corollary ensures the existence of \emph{a} contact structure adapted
to the resulting open book structure on the fibre sum. It does not tell us however,
that the contact structure from the \emph{contact} fibre connected sum is adapted
to this open book. This is what we want to show in the following.
The problem is that the fibres of the symplectic normal bundle to the push-off $B^+$
are not tangent to the pages, i.e.\ the operation of the contact fibre connected
sum, which uses these fibres, does not fit nicely to the open book structure.
We are going to manipulate the abstract open book
(without changing the underlying contact manifold)
in such a way that the symplectic normal fibres of the push-off are tangent to
the pages of the open book and thus guaranteeing compatibility of open book
structure and fibre sum.

\begin{lemma}
\label{lemma:tangentLiouville}
Let $X$ be a codimension $2$ symplectic submanifold of an exact symplectic
manifold $(W, \omega=d\lambda)$ and suppose that the normal bundle of
$X$ is trivial.
Then there is a Liouville form $\widetilde{\lambda}$ such that the corresponding
Liouville vector field is tangent to $X$.
\end{lemma}

\begin{proof}
The submanifold $X \subset W$ is symplectic with symplectic form
$$
\omega' := \omega|_{TX} = (d\lambda)|_{TX} = d(\lambda|_{TX}),
$$
i.e.\ $X$ is exact symplectic and a Liouville form is given by
$\lambda' := \lambda|_{TX}$.
As $X$ is of codimension $2$ and has trivial normal bundle,
the symplectic neighbourhood theorem (see \cite[Theorem~3.30]{McDuff1998})
allows us to write a neighbourhood of $X$ in $W$ as
$X \times D^2$ with symplectic form given as
$\omega = \omega' + sds\wedge d\vartheta$, where $s$ and $\vartheta$ are
polar coordinates on the $D^2$-factor.
In these coordinates one (local) primitive of $\omega$ is given by
$\widetilde{\lambda} := \lambda' + 1/2 s^2 d\vartheta$.
Observe that the restriction of both $\lambda$ and $\widetilde\lambda$ to
$X$ equals $\lambda'$. Hence, by Lemma~\ref{lemma:exact}, we have
$$
\lambda|_{X\times D^2} = \widetilde{\lambda} + dh
$$
for an appropriate function $h$.
Consider a function $\widetilde{h}$ on $W$ which is equal to $h$
near $X$ and vanishes outside a neighbourhood of $X$, and denote its
Hamilton vector field by $X_h$.
If $Y$ is the Liouville vector field corresponding to the Liouville form $\lambda$,
then the sum $Y + X_h$ is again Liouville. The associated Liouville form
restricts to $\widetilde{\lambda}$ near $X$. In particular,
the Liouville vector field is tangent to $X$.
\end{proof}

\begin{remark}
\label{remark:TW_with_changedLiouville}
Also note that the contact structures on the open book obtained by the generalised
Thurston--Winkelnkemper construction performed with two Liouville forms
$\lambda$ and $\lambda + dh$ which coincide near the boundary are isotopic.
Indeed, a family of contact structures is given by using $\lambda + d(sh)$ for
$s\in [0,1]$ and Gray stability can be applied.
\end{remark}

\begin{proposition}
\label{proposition:normals_in_page}
Let $(W, \omega)$ be an exact symplectic manifold
and $X\subset W$ a symplectic submanifold of codimension $2$ with trivial
symplectic normal bundle.
Let $\lambda_t$ ($t\in \mathbb{R}$) be a smooth family of Liouville forms such
that the corresponding Liouville vector fields $Y_t$ are tangent to $X$
and such that the $1$-form $\frac{d}{dt}\lambda_t$ vanishes on the fibres
of the symplectic normal bundle of $X$.

Then the fibres of the conformal symplectic normal bundle of the contact submanifold
$X\times \mathbb{R} \subset (W\times \mathbb{R},
\alpha = \lambda_t + dt)$
coincide
with the fibres of the symplectic normal bundle
of $X \times \{t\}$ in $W \times\{t\}$.
In particular, they are
tangent to the slices $W \times \{t\}$.
\end{proposition}

\begin{proof}
The proof is a straight-forward calculation in a symplectic 
neighbourhood $X\times D^2$ of $X$. See \cite{KoelschDurst} for details.
\end{proof}

We will only need the proposition in a special case during a generalised
Thurston--Winkelnkemper construction.
However, note that it also holds for a family
$$\lambda_t = \big(1-\mu (t)\big)\lambda + \mu(t) \widetilde{\lambda}$$
interpolating two Liouville forms with tangent Liouville vector field provided
that their difference is exact with a primitive function only depending
on $\Sigma'$-directions.
This can then be used when working with Giroux domains (cf.~Section~\ref{GirouxTorsionEx}) instead of the
Thurston--Winkelnkemper construction.

We can now show that the binding sum construction can be made compatible with the underlying contact structures and thus prove Theorem~\ref{thm:contact}.

\begin{proof}[{Proof of Theorem~\ref{thm:contact}}]
By Corollary~\ref{corollary:abstract_nested}
the push-off is an abstract nested open book of an abstract open book description
with page $\Sigma$ and monodromy $\phi$, where $\phi$ restricts to
$$
\phi|_{\Sigma'\times D^2} = \phi' \times \id_{D^2}
$$
for a symplectomorphism $\phi'\colon\thinspace\Sigma'\rightarrow\Sigma'$
in a neighbourhood $\Sigma'\times D^2$ of $\Sigma'$ given by the symplectic normal
bundle.
Furthermore, we can assume that the Liouville form restricts to
$\lambda = \lambda' + 1/2s^2 d\vartheta$ in this neighbourhood
by Lemma~\ref{lemma:tangentLiouville} and its proof and
Remark~\ref{remark:TW_with_changedLiouville}.

Note that the monodromy $\phi$ will not necessarily be exact symplectic but
according to \cite[Lemma~7.3.4]{Geiges2008} it is isotopic through
symplectomorphisms equal to the identity near the boundary to an exact
symplectomorphism.
The idea of the proof is to define a vector field $X$ on $\Sigma$ by the condition
$i_X\omega = \lambda - \phi^*\lambda$ and checking that precomposing $\phi$ with
the time-$1$ flow of $X$ is an exact symplectomorphism with the desired properties.
Now in our situation, observe that the vector field $X$ is tangent to $\Sigma'$,
and moreover, projects to zero under the natural projection
$\Sigma'\times D^2 \rightarrow D^2$
also in a neighbourhood $\Sigma'\times D^2$ as above.
As a consequence, the restriction of the resulting monodromy
(still denoted by $\phi$ by abuse of notation)
will still be of the form
$$
\phi|_{\Sigma'\times D^2} = \phi' \times \id_{D^2}
$$
but now for an \emph{exact} symplectomorphism
$\phi'\colon\thinspace\Sigma'\rightarrow\Sigma'$.
In particular, we have $\phi^*\lambda - \lambda = dh$ for a function $h$ which
only depends on the $\Sigma'$-directions on $\Sigma'\times D^2$.

If we form the generalised mapping torus $\Sigma_h(\phi)$ and equip it with
the contact form $\lambda + dt$, Proposition~\ref{proposition:normals_in_page}
tells us that the fibres of the conformal symplectic normal bundle of the push-off
are tangent to the slices $\Sigma\times\{t\}$ and are given by the $D^2$-direction
of $\Sigma' \times D^2 \subset \Sigma$.
We now want to show that the fibres of the conformal symplectic normal bundle
are then also tangent to the pages in the genuine mapping torus.
For this, it is enough to observe that the diffeomorphism between the
generalised and the genuine mapping tori maps a point $(p, t)$ to a point
$(p, \widetilde{t})$, where $\widetilde{t}$ only depends on $t$ and the value of
the function $h$ at $p$.
Thus, as in a symplectic neighbourhood $\Sigma'\times D^2$ the function $h$ only
depends on $\Sigma'$,
we have that for a point $p'\in \Sigma'$ and fixed $t$
$(p', s, \vartheta, t)$
is mapped to
$(p', s, \vartheta, \widetilde{t})$ independent of $(s, \vartheta)$.
Hence, the fibres of the conformal symplectic normal bundle are tangent to the
pages.
Also, the resulting Reeb vector field is tangent to the push-off, as it is a
multiple of the monodromy vector field outside a neighbourhood of the binding.
Furthermore, by construction it is adapted to the open book structure, i.e.\ it
is transverse to the pages.
Thus, denoting the restriction of the contact form $\alpha$ to the push-off $B^+$
by $\alpha_{B^+}$, the $\alpha$ restricts to
$\alpha_{B^+} + s^2 d\vartheta$
on a tubular neighbourhood $B^+ \times D^2$ (with polar coordinates
$(s, \vartheta)$) given by the conformal symplectic normal bundle of $B^+$.
The fibres of the symplectic normal bundle being tangent to the pages means
that the contact fibre connected sum along $B^+$ has a natural open book structure.
The contact form for the resulting contact structure used in the contact fibre
connected sum construction is of the form
$\widetilde{\alpha} = \alpha_{B^+} + f(s) d\vartheta$
for an appropriate function $f$ and coincides with $\alpha$ near
$B^+ \times \partial D^2$ (see \cite[proof of Theorem~7.4.3]{Geiges2008}).
In particular, the Reeb vector field of $\widetilde{\alpha}$ is still transverse
to the fibres of  the open book fibration. Hence, the resulting contact structure
in the contact fibre connected sum is adapted to the resulting open book structure.
\end{proof}

%%%%%%%%%%%%%%%%%%%%%%%%%%%%%%%%%%%%%%%%%%%%%%%%%%%%%%%%%%%%%%%%%%%%%%%%

\subsection{Examples in the contact setting}
\label{section:contact_examples}

We conclude the paper with some applications and examples of the binding sum construction
in the contact setting.

\subsubsection{$S^4 \times S^1$}

Let $M = M_0 \sqcup M_1$ with $M_i$ the five-dimensional sphere $S^5$ with 
standard contact structure and compatible open
book decomposition $(\Sigma_i = D^4, \id)$.
Then the binding has two components $B_i$, both a standard $3$-sphere that we
can equip with the compatible open book decomposition $(\Sigma'_i = D^2, \id)$.
We have $M_i = D^4 \times S^1 \cup S^3 \times D^2$, so performing
the binding sum on $M$ along the $B_i$ yields $S^4 \times S^1$.
Note that the framing of the binding is unique up to homotopy because
$B$ is simply-connected.
By Theorem~\ref{thm:main1}, the binding sum $S^4 \times S^1$ has a natural open
book decomposition obtained by forming the sum along the push-off of the binding.

Pushing a page $\Sigma'_i = D^2$ of the binding open book into the page
$\Sigma_i = D^4$ and then removing a neighbourhood of $\Sigma_i'$, topologically
turns $\Sigma_i$ into a copy of $D^3 \times S^1$.
The resulting page is then obtained by identifying two copies of $D^3 \times S^1$
along a neighbourhood of $\{\ast\} \times S^1 \subset \partial D^3 \times S^1$,
i.e.\ it is $D^3 \times S^1$ as well.
The new binding is the contact binding sum of the $B_i$ with respect to the
specified open book decomposition.
Hence, the new binding has an open book decomposition with page an annulus and
-- applying the formula from \cite{Klukas} for the page framing --
monodromy isotopic to the identity.
This means that the resulting binding is $S^2\times S^1$ with standard
contact structure.
This has a unique symplectic filling up to blow-up and symplectic equivalence
(cf.~\cite[Theorem~4.2]{Ozbagci2015}), so the resulting page is indeed symplectomorphic to
$D^3 \times S^1$ arising as the $4$-ball with a $1$-handle attached.

\subsubsection{Mori's class of examples}
In \cite{Mori2012} Mori constructs an infinite family of contact structures with
compatible open book decompositions on $S^4 \times S^1$ via a fibre connected
sum construction.
Topologically, it is the sum of two copies of $S^5$ along $3$-spheres as
described in the previous example. However, contact forms and open books are
altered such that the $3$-spheres are no longer the bindings of the open book
decompositions but nested open books. We will briefly explain his construction and
refer the reader to the original article \cite{Mori2012} for details.

Let $S^5 = \{ r_1^2 + r_2^2 +r_3^2 = 1\}$ be the unit sphere in $\mathbb{C}^3$
with standard contact structure as defined by $\alpha = \sum r_i^2 d\varphi_i$
and open book fibration $\pi\colon\thinspace S^5\setminus B \rightarrow S^1$ with
$B := \{r_1 = 0\}$.
For $p$ a positive integer, Mori constructs a function $f$ equal to a positive constant near $0$ such that the $1$-form $\alpha_p := f(r_1)\alpha$ is contact.
Furthermore, there is a function $g$ vanishing around $0$ such that the argument
of the complex-valued function
$$
w_p = g(r_1) r_1^p \textup{e}^{ip\varphi_1} + r_2 r_3 \textup{e}^{i(\varphi_1 + \varphi_2)}
$$
is an open book fibration $\pi_p\colon\thinspace S^5\setminus B_p \rightarrow S^1$
supporting $\ker\alpha_p$, where $B_p$ is the vanishing set of $w_p$.
Note that the restriction of $\pi_p$ to $B$ induces the standard open book with
annular pages of $S^3$. Also, the contact structure on $B$ induced by $\alpha_p$ is up
to scaling with a positive constant the standard contact structure of $S^3$.
Hence, the original binding $B$ is realised as a contact nested open book of
$(S^5, \alpha_p, \pi_p)$.
One can show that the binding $B_p$ of this new open book is diffeomorphic to the
lens space $L(p, p-1)$. Clearly, for $p>1$, $B$ and $B_p$ are not diffeomorphic,
i.e.\ $B$ is an explicit
example of a nested open book which is not a push-off of the binding.

Mori then performs the contact fibre connected sum of two copies of $S^5$
equipped with contact forms $\alpha_{p_i}$ and open books $\pi_{p_i}$ along their
common nested open books $B$.
As the open book fibrations agree in a neighbourhood of the $3$-spheres $B$ in
both copies, this immediately gives rise to an open book structure supporting the resulting
contact structure on $S^4\times S^1$.
Furthermore, the procedure does not only yield an infinite family of contact
structures and leaf-wise symplectic foliations but also to an infinite family of
so-called $\epsilon$-$\tau$-\emph{confoliations} (see~\cite{Mori2012}).

\subsubsection{Open books with Giroux torsion}
\label{GirouxTorsionEx}
Let $M$ be a closed oriented manifold admitting a \emph{Liouville pair}
$(\alpha_+,\alpha_-)$, i.e.\ a pair consisting of a positive contact form $\alpha_+$
and a negative contact form $\alpha_-$ such that
$1/2(e^{-s}\alpha_- + e^{s}\alpha_+)$ is a positively oriented Liouville form on
$\R \times M$ (with $s$ denoting the coordinate on the $\R$-factor).
Then the $1$-form
$$
\lambda_{GT} = \frac{1+ \cos s}{2}\alpha_+ + \frac{1 - \cos s}{2}\alpha_-
+ \sin s dt
$$
defines a positive contact structure on
$\R\times S^1 \times M$ ($s$ and $t$ denote the respective coordinates on the
first two factors) (see \cite[Proposition~8.1]{MNW13}).
With this model, Massot, Niederkr\"uger and Wendl \cite{MNW13} define a
\defterm{Giroux $2k\pi$-torsion domain} as
$([0,2k\pi]\times S^1 \times M, \lambda_{GT})$.
Just as in the $3$-dimensional setting this higher-dimensional version of Giroux
torsion is a filling obstruction in the sense that a contact manifold admitting
a contact embedding of a Giroux $2\pi$-torsion domain is not strongly fillable
(see \cite[Corollary~8.2]{MNW13}).
Observe that a Giroux $2\pi$-torsion domain with boundary blown down
(cf.\ \cite[Section~4]{MNW13})
is the binding sum of two copies
of the open book with page $([0,\pi] \times M, \beta)$ and trivial monodromy,
where
$$\beta = \frac{1}{2}(e^{-s}\alpha_- + e^{s}\alpha_+),$$
along $\{\pi\} \times M$.
Given any contact open book $(\Sigma, \phi)$ with $\Sigma$ having two boundary
components contactomorphic to $M$,
Theorem~\ref{thm:contact} yields an open book decomposition of the binding sum
$(\Sigma, \phi) \,\#\, ([0,\pi] \times M, \id) \,\#\, ([0,\pi] \times M, \id)$,
which is a manifold admitting an embedding of a Giroux $2\pi$-torsion domain
modelled on $M$.

\subsubsection{Fibrations over the circle}
\label{section:circle_fibrations}

Theorems~\ref{thm:main1} and~\ref{thm:contact} yield a contact open book decomposition of certain bundles over the circle with fibres being convex
hypersurfaces.

Recall that an oriented hypersurface $S$ in a contact manifold is called \defterm{convex}
(in the sense of Giroux \cite{Giroux1991}) if there is a contact vector field
transverse to $S$. A neighbourhood of the hypersurface can then be identified with
$S\times\R$ such that the contact structure is $\R$-invariant, i.e.\ there is
a contact form of type $\beta + udt$ with $\beta$ a $1$-form on $S$ and 
$u\colon\thinspace S\rightarrow\R$ a function such that
$(d\beta)^{n-1} \wedge (u d\beta + n\beta \wedge du)$
is a volume form on $S$ (here $2n$ is the dimension of $S$).
Conversely, given a triple $(S,\beta,u)$ consisting of a $(2n)$-dimensional
closed manifold $S$, a $1$-form $\beta$ on $S$ and a function
$u\colon\thinspace S\rightarrow\R$ satisfying the above conditions,
% such that
% $(d\beta)^{n-1} \wedge (u d\beta + n\beta \wedge du)$
% is a volume form on $S$, 
the $1$-form $\beta + u dt$ defines an
$\R$- or $S^1$-invariant contact form on $S\times\R$ or $S\times S^1$,
respectively.
% Therefore, we call such a triple a convex surface as well.
Observe that $S$ with the zero set of $u$ removed is an exact symplectic manifold
with Liouville form $\beta / u$.

Now given a triple $(S,\beta,u)$ as above and a diffeomorphism $\phi$
of $S$ restricting to the identity near $\Gamma := \{u = 0 \}$ and to
symplectomorphisms $\phi_\pm$ on the interior of $S_\pm := \{ \pm u \geq 0 \} $, the $S$-bundle $M$ over $S^1$ with
monodromy $\phi$ carries a natural contact structure, such that each fibre defines a 
convex surface modelled by $(S,\beta,u)$. In particular, the fibration admits a
contact vector field transverse 
to the fibres, which is tangent to the contact structure exactly over $\Gamma$.
Observe that $M$ is equal to the binding sum of the open books $(S_+, \phi_+)$
and $(-S_-, \phi^{-1}_-)$.
Thus, Theorem~\ref{thm:main1} yields an open book
description of $M$, which is adapted to the contact structure by
Theorem~\ref{thm:contact}.

%%%%%%%%%%%%%%%%%%%%%%%%%%%%%%%%%%%%%%%%%%%%%%%%%%%%%%%%%%%%%%%%%%%%%%%%%

%%%%%%%%%%%%%%%%%%%%%%%%%%%%%%%%%%%%%%%%%%%%%%%%%%%%%%%%%%%%%%%%%%

\end{document}